\documentclass[12pt]{amsart}

\usepackage[utf8]{inputenc}
\usepackage[T1]{fontenc}
\usepackage[margin=1.25in]{geometry}
\usepackage[english]{babel}
\usepackage{amsmath,amsthm,amssymb,amsfonts}
\usepackage[justification=centering]{caption}
\usepackage{mathtools}
\usepackage{cases}
\usepackage{tikz}
\usetikzlibrary{patterns.meta}
\usepackage{hyperref}
\usepackage{enumitem}
\usepackage{pgfplots}
\usepackage{subcaption}
\usepackage{ifthen}
\tikzset{
    axis break gap/.initial=0mm
}
\pgfplotsset{compat=1.18}

\usepackage[backend=bibtex, sorting=nyt, giveninits=true]{biblatex}

\bibliography{bibli.bib,SN}

\makeatletter

\definecolor{darkgreen}{RGB}{0,150,0}
\definecolor{lightred}{rgb}{1,0.87,0.87}
\definecolor{lightgreen}{rgb}{0.8,1,0.8}
\hypersetup{  bookmarksnumbered, linkcolor=blue, citecolor=red,
    breaklinks,
  linkbordercolor=lightred,
  citebordercolor=lightgreen,
}

\IfFileExists{graphs.tex}{

  \newcommand{\data}[1]{data-##1}
  \newcommand{\dataSub}[2]{data-##1-##2}
}{
  \newcommand{\data}[1]{data/##1}
  \newcommand{\dataSub}[2]{data/##1/##2}
}

\@ifundefined{varnothing}{}{\let\emptyset=\varnothing}
\@ifundefined{leqslant}{}{\let\le=\leqslant \let\leq=\le}
\@ifundefined{geqslant}{}{\let\ge=\geqslant \let\geq=\ge}

\newcommand{\IR}{{\mathbb R}}
\newcommand{\IN}{{\mathbb N}}
\newcommand{\normal}{{\boldsymbol{n}}}
\newcommand{\phie}{\varphi_{\mathrm{e}}}

\newcommand{\intd}{\,\mathrm{d}}

\newcommand{\longtilde}[1]{{\stackrel{\sim}{\smash{#1}\rule{0pt}{1.1ex}}}}
\newcommand{\abs}[1]{\mathopen| #1\mathclose|}
\newcommand{\intervalcc}[1]{\mathopen[#1\mathclose]}
\newcommand{\intervalco}[1]{\mathopen[#1\mathclose[}
\newcommand{\intervaloc}[1]{\mathopen]#1\mathclose]}
\newcommand{\intervaloo}[1]{\mathopen]#1\mathclose[}

\newcommand{\Bigintervaloc}[1]{\Bigl]#1\Bigr]}

\DeclareMathOperator{\Div}{div}
\DeclareMathOperator{\Dom}{Dom}

\newcommand{\gammac}{\Gamma_{\! \mathrm{c}}}
\newcommand{\gammad}{\Gamma_{\! \mathrm{d}}}
\newcommand{\gammadi}{\Gamma_{\! \mathrm{d},i}}

\newcommand{\rhoc}{\rho_{\mathrm{c}}}
\newcommand{\rhod}{\rho_{\mathrm{d}}}
\newcommand{\rhoch}{\hat{\rho}_{\text{c}}}
\newcommand{\rhot}{\Tilde{\rho_1}}
\newcommand{\rhotn}{\Tilde{\rho}_{1,n}}

\newcommand{\phil}{\varphi_{\raisebox{-1pt}{$\scriptstyle\ell$}}}
\newcommand{\rhol}{\rho_{\raisebox{-1pt}{$\scriptstyle\ell$}}}

\numberwithin{equation}{section}

\newtheorem{theorem}{Theorem}[section]
\newtheorem{lemma}[theorem]{Lemma}

\newtheorem*{algorithm}{Algorithm}

\theoremstyle{remark}
\newtheorem{remark}[theorem]{Remark}

\makeatother

\title[Ion Flow Field problem applied to HVDC transmission lines]{  Multiple solutions and simulations for an Ion Flow Field problem
  applied to HVDC transmission lines}
\date{\today}

\author[M.~Chauvier]{Madeline Chauvier}\address[Madeline Chauvier]{Universit\'e Polytechnique Hauts-de-France, C\'ERAMATHS/DMATHS and FR CNRS 2037,F-59313 - Valenciennes Cedex 9 France and
  D\'epartement de Math\'ematique,
  Universit\'e de Mons,
  place du parc~20,
  B-7000 Mons, Belgium}
\email{Madeline.CHAUVIER@umons.ac.be}
\author[S.~Nicaise]{Serge Nicaise}
\address[Serge Nicaise]{Universit\'e Polytechnique Hauts-de-France, C\'ERAMATHS/DMATHS and FR CNRS 2037,
F-59313 - Valenciennes Cedex 9 France}
\email{Serge.Nicaise@uphf.fr}
\author[C.~Troestler]{Christophe Troestler}
\address[Christophe Troestler]{ D\'epartement de Math\'ematique,
  Universit\'e de Mons,
  place du parc~20,
  B-7000 Mons, Belgium}
\email{Christophe.TROESTLER@umons.ac.be}
\author[J.~Venel]{Juliette Venel}
\address[Juliette Venel]{Universit\'e Polytechnique Hauts-de-France, C\'ERAMATHS/DMATHS and FR CNRS 2037,
F-59313 - Valenciennes Cedex 9 France}
\email{Juliette.Venel@uphf.fr}

\begin{document}
\subjclass{35Q60,35G60,35M32,35A16}

\keywords{Ion flow field, HVDC lines, Nonlinear boundary value problems}
\maketitle

\begin{abstract}
  This paper initiates a mathematical investigation of a PDE model for
  the transport of high voltage direct current via
  overhead lines.  We prove the existence of infinitely many solutions,
  give necessary conditions for existence, explicitly compute the
  continuum of all radial solutions, and develop a new numerical
  algorithm for this problem.
\end{abstract}

\section{Introduction}

Nowadays, electricity is predominantly transmitted over high-voltage
(HV) lines using alternating current (AC) rather than direct current
(DC).  This is because AC power transmission makes it easy to change
the voltage magnitude from low (for electricity usage) to high (for
electricity transportation) voltages using transformers.  Transpor\-ting
electricity at a high voltage minimizes energy loss due to Joule
heating.
However, recent advances in power electronics as well as the
development of renewable energies has sparked interest in HVDC
transmission. In addition, HVDC has many advantages over HVAC, some of
which are highlighted hereafter (for a more thorough discussion, the
reader is referred to \cite{power_loss,interet1,interet2}).  Firstly,
the DC is more suitable than AC for long-distance transmission
overhead lines thanks to lower energy
losses.  Indeed, the longer the overhead lines, the more reactive
power is emitted. The reactive power, which is a parasitic effect
specific to AC systems, limits the capacity to transmit active
power---the real quantity of interest. In DC systems, this parasitic
effect no longer exists.
Consequently, the transmission with DC on overhead lines is also more
economical after the break-even distance (which is approximately
600--800 km), even if one takes into account the converter station for
DC.  Secondly, DC is also preferable for underground and submarine
lines---such as those bringing the energy from off-shore wind
turbines---due to the higher capacitance that affects the transmission
when using AC \cite{ABCs}.  The cost is lower after 40 km,
even when, again, the cost of the DC converter station is taken into account.
These
advantages are very interesting for transporting renewable
energy produced off-shore or at great distances from
cities.
Finally, and particularly relevant to this paper, the corona effect---whereby
some space charges are created---is
greatly reduced with DC~\cite{Wang}.  However, this last effect also
presents some challenges.  For overhead lines, the ions migrate away from
the cables due to the fixed polarity of DC,
which modifies the electric field in the air up to the ground.  For
public health reasons, it is important to quantify the magnitude of the
electric field on the ground to ensure that it is under an
acceptable level.

This subject has already been studied for a couple of years now by
several researchers
\cite{debut_finite,Sarma-Janischewskyj:1969,recent} but, as far as we
know, it has not been investigated from a mathematical point of view.

The simplest mathematical model used by the engineering community
\cite{AmorusoLattarulo:14,
  ButlerCendesHoburg:89,CristinaDinelliFeliziani:91,Lobry:14} is
described by the following three equations in the unipolar case (i.e.,
for a single conductor).
\begin{enumerate}
\item Poisson's equation:
  \begin{equation*}
    -\Delta \varphi = \frac{\rho}{\varepsilon_0},
  \end{equation*}
  where $\varphi$ represents the electric potential, $\rho$ denotes
  the space charges density and $\varepsilon_0$ is the
  permittivity of the air.
\item Ion current equation:
  \begin{equation*}
    J = \bigl(-\mu \nabla \varphi + W\bigr)\rho- D \nabla \rho,
  \end{equation*}
  where $J$ represents the ion current density, and $\mu$, $W$ and $D$ are
  constants representing respectively the ion mobility, the velocity
  of the wind, and the diffusion coefficient.
\item Current continuity equation:
  \begin{equation*}
    \Div J = 0.
  \end{equation*}
\end{enumerate}
These equations are considered on a domain $\Omega$, which corresponds
to the air surrounding the conductors above the ground. Ideally, it
would be unbounded, but for practical reasons it is reduced to a
bounded domain in a plane orthogonal to the transmission line; see
Figure~\ref{fig : half-disk} for an illustration.

\begin{figure}[h!t]
  \centering
  \begin{tikzpicture}[scale = 0.5,baseline={([yshift=-1ex]current bounding box.center)}]
    \draw[fill=black!10] (0:0) arc (0:180:5.5cm) -- (-11,0) -- (0,0);
    \draw[fill=white] (-5.5,2.5) circle (0.8);
    \draw (-5.5,1.1) node {$\gammac$};
    \draw (-1,0.57) node {$\gammad$};
    \draw (-8.5,2) node {$\Omega$};
    \fill[pattern={Lines[angle=45,distance=5pt]}] (-11,0) rectangle (0,-1);  \end{tikzpicture}
  \caption{A half-disc with a circular conductor.}
  \label{fig : half-disk}
\end{figure}
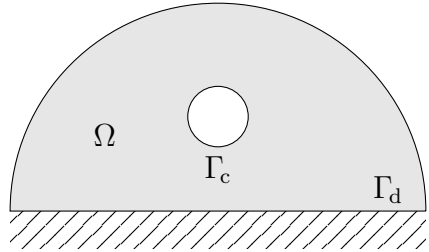

To maintain a certain
degree of generality, we assume that the domain $\Omega$ is a bounded
and connected set of $\IR^2$ such that
$\partial \Omega = \gammac \cup \gammad$ and
$\gammac \cap \gammad = \emptyset$, $\gammac$ representing the
boundary of the conductor, and $\gammad$ the boundary of the air
region and the contact with the ground (see Figure \ref{fig:omega} for an illustration).  We
here assume that $\gammac$ is $\mathcal{C}^{1,1}$, while $\gammad$ is
a curvilinear polygon of class $\mathcal{C}^{1,1}$
(see~\cite[Definition~1.4.5.1]{grisvard})
with no re-entrant corner in
the sense that $\gammad = \bigcup_{i=1}^I\gammadi$, $I \in \IN^{\ge 1}$, and
$\gammadi $ is $\mathcal{C}^{1,1}$ with the interior angle between two
consecutive $\gammadi$ is less than $\pi$. For simplicity, we
consider $W=D=0$.  In addition, without loss of
generality, we can normalize the following constants:
$\varepsilon_0 = \mu=1$.
With these choices, the previous equations become:
\begin{equation}\label{sys:gen}
  \begin{cases}
    -\Delta \varphi = \rho , &  \text{in } \Omega,\\[2\jot]
    \Div \bigl(\rho  \nabla\varphi \bigr) = 0, & \text{in } \Omega,
  \end{cases}
\end{equation}
where $\varphi$ and $\rho$ are the unknowns defined in the domain $\Omega$.

To solve this problem, we need to add some boundary
  conditions. It is standard practice  \cite{AmorusoLattarulo:14,ButlerCendesHoburg:89,CristinaDinelliFeliziani:91,Lobry:14}
to set the potential $\varphi$ to a constant $V$ on the
conductor and to $0$ on the ground.
Without loss of generality, we can set the potential to $1$ on
the conductor.
To represent the corona effect,
physicists commonly fix the outer normal derivative of the
potential on the conductor to a function $A: \gammac \to \IR$.
This is called the Kaptzov's
assumption \cite{debut_finite,onset_formula,diff2018}.  The boundary
conditions are thus given by:
\begin{equation} \label{sys:BC}
  \begin{cases}
    \varphi = 1, & \text{ on } \gammac,\\
    \varphi = 0, & \text{ on } \gammad, \\[2\jot]
    \displaystyle \frac{\partial \varphi}{\partial \normal} =  A,
                 & \text{ on } \gammac.
  \end{cases}
\end{equation}
Let us note that no boundary condition is imposed on $\rho$, one is
imposed on $\varphi$ on $\gammad$, and two are imposed on $\varphi$ on
$\gammac$.  Altogether, this yields three boundary conditions on
$\varphi$.  This can be explained by the fact that using the first
equation to eliminate $\rho$ in the second one of \eqref{sys:gen}
yields
\begin{math}
  \Div \bigl((\Delta \varphi)  \nabla\varphi \bigr) = 0,
  \text{ in } \Omega,
\end{math}
which is a nonlinear partial differential equation of order three in
$\varphi$.
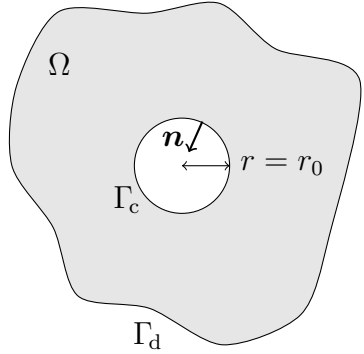
\begin{figure}[h!t]
  \centering
  \begin{tikzpicture}[baseline = 1ex, scale = 0.9]
    \path
    \foreach \X/\Y in {1/0.9,2/0.2,3/0.45,4/0.3,5/0.9,6/0.7,7/0.5,8/0.1,9/0.45,10/0.15,11/0.7,12/0.8,13/0.4}
    {(360*\X/14:2+\Y) coordinate (A\X)};
    \draw[fill=black!10] plot[smooth cycle,samples at={1,...,13}] (A\x);
    \draw[fill=white] (0,0) circle (0.7);
    \draw[<->] (0,0) -- (0.7,0) node[right]{$r=r_0$};
    \node[left] at (-0.45,-0.5) {$\gammac$};
    \node at (-0.5, -2.5) {$\gammad$};
    \node at (-1.8,1.5) {$\Omega$};
    \draw[->, thick] (0.3,0.65) -- (0.1,0.2);
    \node[left] at (0.2,0.4) {$\normal$};
  \end{tikzpicture}
  \caption{A domain $\Omega$ around a circular conductor.}
  \label{fig:omega}
\end{figure}

\vspace*{2ex}

A trivial solution to Problem~\eqref{sys:gen}--\eqref{sys:BC} for a
suitably chosen function $A$ is the \emph{electrostatic solution},
which corresponds to the case of a vanishing charge density, i.e.\
$\rho \equiv 0$.  We denote the corresponding potential $\phie$, that
is the unique solution to
\begin{equation}\label{sys:elec}
  \begin{cases}
    -\Delta \phie = 0, &  \text{in } \Omega,\\
    \phie = 1, & \text{on } \gammac,\\
    \phie = 0,& \text{on } \gammad.
  \end{cases}
\end{equation}
The aim of this paper is to initiate a mathematical investigation of
Problem~\eqref{sys:gen}--\eqref{sys:BC} for $\rho \ge 0$ and
$\rho \not\equiv 0$.  In
section~\ref{sdiffusion}, we use a fixed point argument to prove the
existence of a solution with an additional small
diffusion term (which is present in some physical models
\cite{diff2015,diff2018,diff2013, diff2020}) and with the Neumann condition from \eqref{sys:BC}
replaced by a Dirichlet boundary condition on $\rho$ on the
whole boundary $\partial\Omega$.  Passing to the limit as the
diffusion coefficient
tends to zero, we show in section~\ref{s:unipolar} the existence of
infinitely many
solutions $(\varphi,\rho)$, with $\rho \not\equiv 0$, to
problem~\eqref{sys:gen} with the Dirichlet boudary conditions from
\eqref{sys:BC}.  However, the Neumann condition from \eqref{sys:BC} is again omitted and, by this limit procedure,
the Dirichlet
boundary condition on $\rho$ is lost. In section \ref{s:properties},
we first use the maximum principle to establish some comparison
between the normal derivatives of the solution to \eqref{sys:gen} with
the Dirichlet boundary conditions from \eqref{sys:BC} and of the
electrostatic solution.  We also give bounds on the normal
derivative of the electrostatic solution $\phie$ in the particular case where
$\gammac$ is circular. This enables us to deduce some necessary conditions
on $A$ for a solution to \eqref{sys:gen}--\eqref{sys:BC} to exist.
In section~\ref{s:radial}, we explicitly determine all radial 
solutions to
\eqref{sys:gen}--\eqref{sys:BC} when $\Omega$ is an annulus and $A$ is
radial (whence constant).  Some of these solutions 
were already mentioned without proof in the electrical engineering
community
\cite{Janischewskyj-Cela:1979,Sarma-Janischewskyj:1969,Takuma-Ikeda-Kawamoto:1981,Townsend:1914}
who used them as a benchmark for their algorithms.
We also do the same in section~\ref{snumerique} to provide evidence for the
convergence of a proposed new numerical algorithm.

\medskip

To conclude this introduction, let us introduce some notation used
throughout the paper.  The usual norm and semi-norm of
$H^{s}(\Omega)$, $s\ge 0$, are denoted by $\|\cdot\|_{s,\Omega}$ and
$|\cdot|_{s,\Omega}$, respectively. For $s=0$ we drop the index
$s$. The scalar product in $L^2(\Omega)$ is denoted by
$(\cdot \mid \cdot)_{\Omega}$. The same notation will be used for vector valued functions.
The Fréchet differential of a function $J$ of the
variable $\rho$ at a point $\rho_0$ in the direction $z$ is denoted by
$(\partial J/ \partial \rho) (\rho_0) [z]$.  Finally,
$X \hookrightarrow Y$ means that the Banach space $X$ is continuously
embedded into the Banach space $Y$.
For the curvilinear polygon $\gammad = \bigcup_{i=1}^I\gammadi$, we
denote $W^{2-1/p,p}(\gammad)$ with $p > 2$ the space of all continuous
functions $u$ defined on $\gammad$ such that
$u|_{\gammadi} \in W^{2-1/p,p}(\gammadi)$ for all $i$
(see \cite[Theorem~1.5.2.8]{grisvard} for the general compatibility
conditions).

\section{Existence of a solution for the unipolar case with diffusion}
\label{sdiffusion}

To prove an existence result for the problem \eqref{sys:gen} with
the Dirichlet boundary conditions from \eqref{sys:BC}, we first add a diffusion
term to regularize the problem.  In that case, inspired from
existence results for the drift diffusion model (see
\cite[\S~3.2]{Markowich:86} for instance), using a fixed point
argument, an existence result is available.

\begin{theorem}
  \label{thm : existence diff}
  Let
  $\rhoc \in H^{1/2}(\gammac)$, $\rhod \in H^{1/2}(\gammad)$,
  $\rhoc,\rhod \geq 0$ and define
  $K_+ \coloneq \max \lbrace \sup_{\gammac}\rhoc, \sup_{\gammad} \rhod
  \rbrace$ which is supposed to be strictly positive. Let us
  introduce
  \begin{equation*}
    W \coloneq \lbrace \rho  \in L^2(\Omega) \mid
    0 \le \rho \le K_+\rbrace.
  \end{equation*}
  Then for all $\varepsilon >0$, there exists a solution
  $(\varphi_\varepsilon, \rho_\varepsilon) \in H^{2}(\Omega) \times
  \bigl(W \cap H^1(\Omega)\bigr)$ to
  \begin{equation} \label{sys:sys_complet}
    \begin{cases}
      -\Delta \varphi_\varepsilon = \rho_\varepsilon, &  \text{ in } \Omega,  \\
      \Div \bigl( \varepsilon \nabla \rho_\varepsilon + \rho_\varepsilon \nabla\varphi_\varepsilon\bigr) = 0,
                              & \text{ in } \Omega,\\
      \varphi_\varepsilon = 1, &  \text{ on } \gammac,\\
      \varphi_\varepsilon = 0,&  \text{ on } \gammad, \\
      \rho_\varepsilon = \rhoc, &  \text{ on } \gammac, \\
      \rho_\varepsilon = \rhod, &  \text{ on } \gammad.
    \end{cases}
  \end{equation}
  Moreover, $\varphi_\varepsilon \in W^{2,q}(\Omega)$ for every $q \ge 2$.
  Finally, if $\rhoc \in W^{2-1/p,p}(\gammac)$
  and $\rhod \in W^{2-1/p,p}(\gammad)$ for some $p > 2$, then
  $\rho_\varepsilon \in W^{2,p}(\Omega)$.
\end{theorem}

The idea of the proof is to transform System~\eqref{sys:sys_complet}
into a fixpoint $\rho = G(\rho)$ by splitting it into two simpler
equations.  More precisely, given $\rho_0 \in W$, solve
\begin{equation} \label{sys:sys_phi}
  \begin{cases}
    -\Delta \varphi = \rho, &  \text{in } \Omega,  \\
    \varphi = 1,& \text{on }\gammac,\\
    \varphi = 0,& \text{on }\gammad,
  \end{cases}
\end{equation}
with $\rho = \rho_0$ for the unknown is $\varphi$, and with that
solution $\varphi_0$ solve
\begin{equation} \label{sys:sys_rho}
  \begin{cases}
    \Div \bigl( \varepsilon \nabla \rho_1+ \rho_1 \nabla\varphi_0\bigr) = 0,
    &  \text{ in } \Omega,\\
    \rho_1 = \rhoc, &  \text{on } \gammac, \\
    \rho_1 = \rhod, &  \text{on } \gammad,
  \end{cases}
\end{equation}
for $\rho_1$.  Set $G(\rho_0) \coloneq \rho_1$.

Before detailing the proof of Theorem~\ref{thm : existence diff},
let us start with a few preparation lemmas.

\begin{lemma}
  \label{lemma : bdvarphi0}
  \label{lemma:first-eq}
  The linear map $W \to H^1(\Omega): \rho \mapsto \varphi$ where
  $\varphi$ is the unique solution to~\eqref{sys:sys_phi}
  is well defined and continuous from $W$ endowed with the
  $L^2$-topology to $W^{2,q}(\Omega)$ for any $q \ge 2$.  Moreover its
  image is bounded in the sense that, there exists positive constants
  $C_q$, for every $q \ge 2$, and $C_\infty$ such that, for all
  $\varphi$ in the image, $\varphi \in C^1(\overline{\Omega})$,
  \begin{gather}
    \label{eq : phi bornée}
    \|\varphi\|_{W^{2,q}(\Omega)} \le C_q,
    \rlap{\quad\text{and}}\\
    \label{eq : grad phi bound}
    \|\nabla \varphi\|_{L^{\infty}(\Omega; \IR^2)}
    \le C_\infty.
  \end{gather}
\end{lemma}

The proof of this statement is quite standard but we sketch it briefly
for the reader's convenience.

\begin{proof}
  Thanks to \cite[Theorem~5.2.7]{grisvard}, there exists a unique
  $\varphi \in H^2(\Omega)$ satisfying~\eqref{sys:sys_phi}.
  Furthermore as $\rho \in L^\infty(\Omega)$, we also have
  $\varphi \in W^{2,q}(\Omega)$ for every $q \ge 2$.
    As a consequence, the electrostatic solution $\phie \in H^1(\Omega)$
  defined by \eqref{sys:elec} also belongs to $W^{2,q}(\Omega)$, for
  all $q \ge 2$.  Hence the difference $\varphi - \phie$ satisfies
  \begin{equation*}
    \begin{cases}
      -\Delta (\varphi - \phie) = \rho, &  \text{in } \Omega,  \\
      \varphi - \phie = 0,& \text{on } \partial\Omega
                            = \gammac \cup \gammad.
    \end{cases}
  \end{equation*}
  Thus there exists a constant $\tilde C_q>0$ such that
  \begin{equation*}
    \|\varphi - \phie\|_{W^{2,q}(\Omega)}
    \le \tilde C_q \|\rho\|_{L^q(\Omega)}
    \le \tilde C_q K_+ |\Omega|^{1/q},
  \end{equation*}
  where the last inequality results from the fact that $\rho$ belongs
  to $W$.  Thus the continuity is established.  In addition, estimate
  \eqref{eq : phi bornée} directly follows from the triangular
  inequality with
  $C_q \coloneq \|\phie\|_{W^{2,q}(\Omega)} + \tilde C_q K_+
  |\Omega|^{1/q}$.

  Finally, from the continuous embedding
  $W^{2,q}(\Omega) \hookrightarrow \mathcal{C}^1(\overline{\Omega})$ for
  $q > 2$, we deduce that $\nabla \varphi$ is in
  $L^\infty(\Omega; \IR^2)$ with
  \begin{equation}
    \|\nabla \varphi\|_{L^{\infty}(\Omega; \IR^2)}
    \le C_\infty \coloneq K_q C_q,
  \end{equation}
  where $K_q$ is the norm of the embedding of $W^{2,q}(\Omega)$ into
  $\mathcal{C}^1(\overline{\Omega})$.
\end{proof}

\begin{lemma}
  \label{lemme : rho_reg}
  Under the assumptions of Theorem \ref{thm : existence diff}, for
  $\rho_0 \in W$, there exists a unique solution
  $\rho_1 \in W \cap H^1(\Omega)$ to \eqref{sys:sys_rho} with
  $\varphi_0$ a solution to \eqref{sys:sys_phi}.
    Moreover, if $\rhoc \in W^{2-1/p,p}(\gammac)$ and
  $\rhod \in W^{2-1/p,p}(\gammad)$ for some $p > 2$, then
  $\rho_1 \in W^{2,p}(\Omega)$.
\end{lemma}
\begin{proof}
  We are looking for a solution $\rho_1$ of \eqref{sys:sys_rho}.
  Due to the non homogeneous boundary conditions, we first consider
  a lifting $r \in H^1(\Omega)$ of the boundary
  data \cite[Theorem~1.5.1.3]{grisvard}, namely
  such that
  \begin{equation}\label{eq : r}
    \begin{cases}
      r = \rhoc, & \text{ on } \gammac, \\
      r = \rhod, & \text{ on } \gammad.
    \end{cases}
  \end{equation}
  Now, we are looking for $\rhot \coloneq \rho_1 -r$ in
  $H^1_0(\Omega)$ which satisfies
  \begin{equation}
    \label{eq:tilde-rho1}
    \Div \bigl( \varepsilon \nabla \Tilde{\rho_1}
    + \Tilde{\rho_1} \nabla\varphi_0\bigr)
    = - \Div \bigl( \varepsilon \nabla r + r \nabla\varphi_0\bigr),
    \text{ in } \Omega,
  \end{equation}
  or in the variational form
  \begin{equation}\label{eq : var rho}
    \forall \chi \in H^1_0(\Omega), \quad
    \int_\Omega \bigl(\varepsilon \nabla \Tilde{\rho_1}
    + \Tilde{\rho_1} \nabla \varphi_0\bigr) \nabla \chi \intd x
    = -\int_\Omega \bigl( \varepsilon \nabla r
    + r \nabla\varphi_0\bigr) \nabla \chi \intd x .
  \end{equation}
  We assert that the bilinear form $a$ defined by
  \begin{equation}
    \label{def:a}
    a : H^1_0(\Omega)\times H^1_0(\Omega) \to \IR: 
    (\rho,\chi) \mapsto
    \int_\Omega  \bigl(\varepsilon \nabla \rho
    + \rho \nabla \varphi_0\bigr) \nabla \chi \intd x,
  \end{equation}
  is continuous and coercive. Indeed, the continuity stems from
  Cauchy-Schwarz's inequality by using
  $\nabla \varphi_0 \in W^{1,q}(\Omega;\mathbb{R}^2)$ and the continuous embedding
  $W^{1,q}(\Omega;\mathbb{R}^2) \hookrightarrow \mathcal{C}(\bar{\Omega};\mathbb{R}^2)$ for
  $q>2$. Furthermore, by Green's formula, for
  $\chi \in H^1_0(\Omega)$, we have
  \begin{align*}
    \int_\Omega \chi \nabla \varphi_0 \nabla \chi \intd x
    & = - \int_\Omega \Div \bigl(\chi \nabla \varphi_0 \bigr) \chi \intd x \\
    & = - \int_\Omega \nabla \chi \nabla \varphi_0 \chi \intd x
      - \int_\Omega \chi^2 \Delta \varphi_0 \intd x.
  \end{align*}
  Since $-\Delta \varphi_0 = \rho_0$, it follows that
  \begin{equation*}
    2 \int_\Omega \chi \nabla \varphi_0 \nabla \chi \intd x
    = \int_\Omega \chi^2 \rho_0 \intd x \geq 0.
  \end{equation*}
  So
  \begin{equation}\label{eq : a}
    a(\chi,\chi) = \int_\Omega \varepsilon \lvert \nabla \chi\rvert^2
    + \tfrac{1}{2}\rho_0 \chi^2 \intd x.
  \end{equation}
  Therefore $a$ is coercive in $H^1_0(\Omega)$ and, thanks to the
  Lax-Milgram theorem, there exists a unique solution
  $\rhot \in H^1_0(\Omega)$ of \eqref{eq : var rho}. Hence, we also
  conclude that there exists a unique $\rho_1 \in H^1(\Omega)$
  solution of \eqref{sys:sys_rho}.

  We also need to show that $\rho_1 \in W$.  Notice that we can
  apply the maximum principle \cite[Theorem~8.1]{gilbarg2001elliptic}
  to the system \eqref{sys:sys_rho} because
  $\Div \nabla \varphi_0 = -\rho_0 \leq 0$ (see condition~8.8 in
  \cite{gilbarg2001elliptic}).  Therefore
  \begin{equation*}
    0 \leq \rho_1 \leq K_+,
  \end{equation*}
  because $\rho^-_{\text{c}} = \rho^-_{\text{d}} = 0$.  Thus
  $\rho_1 \in W \cap H^1(\Omega)$.

  Now let us assume that $\rhoc \in W^{2-1/p,p}(\gammac)$ and
  $\rhod \in W^{2-1/p,p}(\gammad)$ for some $p > 2$ and
  show that $\rho_1 \in W^{2,p}(\Omega)$.
  First note that one can choose $r \in W^{2,p}(\Omega)
  \hookrightarrow H^2(\Omega)$
  (see~\cite[Theorem~1.5.2.8]{grisvard} where it is standard that the
  compatibility conditions can be satisfied with a suitable choice
  of the normal derivative on the boundary $\gammad$).
  Expanding \eqref{eq:tilde-rho1} and using that $-\Delta \varphi_0 =
  \rho_0$ yields 
  \begin{equation}\label{eq : rho_reg}
    \Div(\varepsilon \nabla \Tilde{\rho_1})
    = \rhot \rho_0 - \nabla \rhot \nabla \varphi_0
    -\varepsilon \Delta r - \nabla r \nabla \varphi_0
    + r \rho_0.
  \end{equation}
  Since $\nabla \varphi_0 \in L^{\infty}(\Omega; \mathbb{R}^2)$ and $\rho_0 \in L^\infty(\Omega)$, we
  deduce from \eqref{eq : rho_reg} that $\Delta \rhot \in L^2(\Omega)$
  and so $\rhot \in H^2(\Omega)$ with the help of \cite[Theorem
    5.2.7]{grisvard}.  Due to the continuous embedding
  $H^1(\Omega;\mathbb{R}^2) \hookrightarrow L^p(\Omega;\mathbb{R}^2)$,
  $\nabla \rhot \in L^p(\Omega;\mathbb{R}^2)$. Therefore,  Equality \eqref{eq : rho_reg} implies that $\Delta \rhot \in L^p(\Omega)$ and so that
  $\rhot \in W^{2,p}(\Omega)$, again by
  \cite[Theorem~5.2.7]{grisvard}.  Thus $\rho_1 \in W^{2,p}(\Omega)$.
\end{proof}

\begin{lemma}\label{lemma : G cont}
  Under the assumptions of Theorem \ref{thm : existence diff},
  the functional $G$ is completely continuous from $W$ to $W$.
\end{lemma}
\begin{proof}
  Let $(\rho_{0,n})_{n\in \IN}$ be a bounded sequence included in $W$.
  For every $n$, let
  $\varphi_{0,n}$ be the solution to \eqref{sys:sys_phi} with
  $\rho_{0,n}$ instead of $\rho$ and
  $\rho_{1,n} = G(\rho_{0,n})$ be the solution to \eqref{sys:sys_rho} with
  $\varphi_{0,n}$ instead of $\varphi_0$.
  Let us prove that, up to a subsequence, $(\rho_{1,n})_n$ strongly
  converges in~$W$.

  First, there exists a subsequence of $(\rho_{0,n})_n$, still
  denoted by $(\rho_{0,n})_n$, and $\rho_0 \in L^2(\Omega)$ such
  that
  \begin{equation*}
    \rho_{0,n} \rightharpoonup \rho_0
    \hspace*{2ex} \text{ weakly in } L^2(\Omega),
  \end{equation*}
  with $\rho_0 \in W$ since $W$ is convex, whence weakly closed.

  Secondly, Lemma~\ref{lemma:first-eq} 
  implies that the sequence $(\varphi_{0,n})_{n}$ is
  bounded in $H^{2}(\Omega)$. Hence by the compact embedding of
  $H^{2}(\Omega)$ into $H^s(\Omega)$, for any $s\in \intervalco{0,2}$,
  passing if necessary to a subsequence,
  there exists $\varphi_0 \in H^s(\Omega)$,
  such that
  \begin{equation} \label{eq : strongtildephi1n}
    \varphi_{0,n} \rightarrow \varphi_0
    \hspace*{2ex} \text{ strongly in } H^s(\Omega)
    \text{ for all } s \in \intervalco{0, 2}.
  \end{equation}
  Using the variational formulation of \eqref{sys:sys_phi}, we
  deduce that $\varphi_0$ is solution of \eqref{sys:sys_phi} with
  $\rho = \rho_0$ and
  Lemma~\ref{lemma : bdvarphi0} implies that
  $\|\nabla\varphi_0\|_{L^{\infty}(\Omega; \IR^2)} \le C_\infty$.

  Thirdly, let $\rhotn \coloneq \rho_{1,n}-r$ for every $n \in \IN$
  where $r \in H^1_0(\Omega)$ is defined by \eqref{eq : r}. In
  \eqref{eq : var rho}, with $\rhotn$ and $\chi = \rhotn$, we will
  have, thanks to Cauchy-Schwarz's inequality, \eqref{eq : a} and
  \eqref{eq : grad phi bound},
  \begin{align*}
    \varepsilon \int_\Omega \lvert \nabla \Tilde{\rho}_{1,n}\rvert^2 \intd x
    & \le a(\rhotn,\rhotn)
      = - \int_\Omega \bigl(\varepsilon \nabla r
      + r \nabla \varphi_{0,n}\bigr) \nabla \Tilde{\rho}_{1,n} \intd x \\
    &\le \varepsilon \|\nabla r\|_\Omega \,
      \|\nabla \Tilde{\rho}_{1,n}\|_\Omega
      + C_\infty \|r\|_\Omega
      \| \nabla \Tilde{\rho}_{1,n}\|_\Omega.
  \end{align*}
  Therefore, $(\rhotn)_n$ is bounded in $H^1_0(\Omega)$, thus, up to
  a subsequence, there exists $\rhot \in H^1_0(\Omega)$,
  \begin{equation*}
    \Tilde{\rho}_{1,n} \rightharpoonup \Tilde{\rho}_1
    \hspace*{2ex} \text{ weakly in } H^1_0(\Omega),
  \end{equation*}
  and
  \begin{equation}\label{eq : strongtilderuo1n}
    \Tilde{\rho}_{1,n} \rightarrow \Tilde{\rho}_1
    \hspace*{2ex} \text{ strongly in } H^t(\Omega),
    \text{ for all } t \in \intervalco{0, 1}.
  \end{equation}

  To show that $\rho_1$ is a solution of \eqref{sys:sys_rho}, let us
  establish that $\rhot$ is solution of \eqref{eq : var rho}.  We know
  that
  \begin{equation} \label{eq : int}
    \forall \chi\in H^1_0(\Omega),\quad
    \int_\Omega \bigl(\varepsilon \nabla \Tilde{\rho}_{1,n}
    + \Tilde{\rho}_{1,n} \nabla \varphi_{0,n}\bigr)
    \nabla \chi \intd x
    = -\int_\Omega \bigl(\varepsilon \nabla r
    + r \nabla \varphi_{0,n}\bigr) \nabla \chi \intd x.
  \end{equation}
  But, by \eqref{eq : strongtildephi1n} and \eqref{eq : strongtilderuo1n} and the help of
  \cite[Theorem~1.4.4.2]{grisvard}, we deduce that
  \begin{equation*}
    \Tilde{\rho}_{1,n} \nabla \varphi_{0,n} \rightarrow
    \Tilde{\rho}_1 \nabla \varphi_0
    \hspace*{2ex} \text{ strongly in } L^2(\Omega;\IR^2).
  \end{equation*}
  Consequently, because $(\rhotn)_n$ converges weakly in $H^1_0(\Omega)$
  to $\rhot$ and $(\varphi_{0,n})_n$ converges strongly to $\varphi_0$ in
  $H^1(\Omega)$, it is possible to take the limit in equation
  \eqref{eq : int} . Therefore, we deduce that $(\rho_{1,n})_n$
  strongly converges in $L^2(\Omega)$ to
  $\rho_1 = \rhot +r = G(\rho_0) \in W$ solution of \eqref{sys:sys_rho}. By Lemma \ref{lemme : rho_reg}, we also have
  $\rho_1 \in W^{2,p}(\Omega) \cap W$.
\end{proof}

\begin{proof}[\proofname\ of Theorem~\ref{thm : existence diff}]
  As mentioned previously, we are using a fixed point argument to
  solve this system.  To this end, we first split it into two
  subsystems.   Given $\rho_0 \in W$, we successively solve
  \eqref{sys:sys_phi} with $\rho = \rho_0$ to get $\varphi_0$ and
  then~\eqref{sys:sys_rho}
  for $\rho_1$ and define $G(\rho_0) \coloneq \rho_1$.

  The existence and regularity of $\varphi_0$ is given by
  Lemma~\ref{lemma:first-eq}.
  Given such a $\varphi_0$, Lemma \ref{lemme : rho_reg} implies that
  Problem \eqref{sys:sys_rho} has a unique solution
  $\rho_1 \in H^1(\Omega) \cap W$ and
  thanks to Lemma \ref{lemma : G cont}, $G$ is completely continuous
  from $W$ to $W$.  Hence by Schauder's fixed point theorem
  \cite[Theorem~1.2.3]{schauder}, $G$ has a fixed point in $W$. In
  other words, there exists a solution
  $(\varphi_\varepsilon, \rho_\varepsilon)$ of \eqref{sys:sys_complet} with $\varphi_\varepsilon \in W^{2,q}(\Omega)$ for
  every $q \geq 2$ and $\rho_\varepsilon \in W \cap H^1(\Omega)$.

  Finally, if $\rhoc \in W^{2-1/p,p}(\gammac)$ and
  $\rhod \in W^{2-1/p,p}(\gammad)$, Lemma~\ref{lemme : rho_reg}
  gives the desired regularity on $\rho_\varepsilon$.
\end{proof}

\section{Existence of infinitely many solutions in a unipolar case}
\label{s:unipolar}

Thanks to Theorem \ref{thm : existence diff} of the previous
section we can show the existence of a nontrivial
solution $(\varphi,\rho)$ to
the following system:
\begin{equation} \label{sys:sys_sans_A}
  \begin{cases}
    -\Delta \varphi = \rho, & \text{in } \Omega,  \\
    \Div(\rho\nabla \varphi) = 0, & \text{in } \Omega,  \\
    \varphi = 1, &  \text{on } \gammac,\\
    \varphi = 0,&  \text{on } \gammad.
  \end{cases}
\end{equation}
Note that this system is not exactly the one we want to solve because
it does not contain the Neumann boundary condition (see the
boundary conditions \eqref{sys:BC} in the introduction). But
once a solution $(\varphi,\rho)$ of \eqref{sys:sys_sans_A} is
known, we can say that it is a solution of
\eqref{sys:gen}--\eqref{sys:BC} for some $A$ being given by
${\partial \varphi}/{\partial\normal}$ (a non-constant
function in general).

\begin{theorem}
  \label{thm : existence_unipolar}
  Problem~\eqref{sys:sys_sans_A} possesses infinitely many nontrivial
  solutions.
  More precisely, for all $n \in \IN$, there exists a solution
  $(\varphi^*_n, \rho^*_n)\in H^2(\Omega)\times W$
  to~\eqref{sys:sys_sans_A} with $\rho^*_n \not \equiv 0$,
  $\varphi^*_n\in W^{2,q}(\Omega)$ for every $q>2$, and
  $\|\rho^*_n\|_{L^\infty(\Omega)} \xrightarrow[n\to \infty]{} 0$.
\end{theorem}

\begin{proof}
  Let us first show the existence of one nontrivial solution.
  For every positive integer $\ell$,
  let us denote by $(\phil, \rhol)$ a solution given by
  Theorem~\ref{thm : existence diff} for $\varepsilon = 1/\ell$,
  $\rhoc(x) = \rhoch$ where $\rhoch > 0$ is a fixed constant, and
  $\rhod \equiv 0$.  Note that $K_+ = \rhoch$.  Because
  $(\rhol)_\ell \subseteq W$, it is bounded in
  $L^2(\Omega)$, so, going if necessary to a subsequence, there exists
  $\rho \in L^2(\Omega)$ such that
  \begin{equation} \label{eq : weak cvg rho}
    \rhol \rightharpoonup \rho
    \quad \text{ weakly in } L^2(\Omega), \quad \text{ as } \ell \rightarrow +\infty.
  \end{equation}
  Note that $\rho \in W$.  Lemma~\ref{lemma:first-eq} implies that
  $\phil \rightharpoonup \varphi$ weakly in $W^{2,q}(\Omega)$
  for all $q \ge 2$.
  Therefore, for any $s\in \intervalco{0,2}$, $(\phil)_\ell$ converges strongly to
  $\varphi$ in $H^s(\Omega)$ whenever $\ell \to \infty$.  In
  particular, $\nabla \phil \to \nabla\varphi$ in $L^2(\Omega ; \mathbb{R}^2)$.
  As in the proof of Lemma \ref{lemma : G cont}, we can deduce that
  $\varphi$ is the solution to \eqref{sys:sys_phi}.
  
  Now, let us prove that
  $\Div(\rho \nabla \varphi)=0$.
  For any fixed $\chi \in L^2(\Omega)$, we have
  \begin{equation*}
    \bigl(\rhol \nabla \phil - \rho\nabla \varphi \bigm| \chi \bigr)_\Omega
    = \bigl(\rhol (\nabla \phil - \nabla \varphi) \bigm| \chi \bigr)_\Omega
    + \bigl((\rhol-\rho)\nabla \varphi \bigm| \chi \bigr)_\Omega.
  \end{equation*}
  The first term tends to $0$ because
  \begin{equation*}
    \bigl| (\rhol (\nabla \phil -\nabla \varphi)\lvert \chi)_\Omega \bigr|
    \le K_+ \|\nabla \phil-\nabla \varphi\|_\Omega \|\chi\|_\Omega
    \xrightarrow[\ell \to \infty]{} 0.
  \end{equation*}
  In addition,
  since
  $\nabla \varphi \in W^{1,q}(\Omega;\mathbb{R}^2) \hookrightarrow
  L^\infty(\Omega;\mathbb{R}^2)$, for $q>2$, we have
  $\chi \nabla \varphi \in L^2(\Omega; \IR^2)$.  From~\eqref{eq :
    weak cvg rho}, we deduce that the second term also tends to
  $0$. In other words
  \begin{equation}\label{eq : cvg weak rho phi}
    \rhol \nabla \phil \rightharpoonup \rho\nabla \varphi
    \hspace*{2ex} \text{ weakly in } L^2(\Omega;\mathbb{R}^2).
  \end{equation}

  Since $(\phil,\rhol)$ is a solution to \eqref{sys:sys_complet},
  we have for all $\chi \in \mathcal{D}(\Omega)$ \footnote{$\mathcal{D}(\Omega)$ is the set of smooth functions with a compact support included in $\Omega$. },
  \begin{equation*}
    0 =
    \bigl(\tfrac{1}{\ell} \nabla \rhol + \rhol \nabla \phil
    \bigm| \nabla \chi\bigr)_\Omega
    = -\tfrac{1}{\ell} \bigl(\rhol \bigm| \Delta \chi\bigr)_\Omega
    + \bigl(\rhol \nabla \phil \bigm| \nabla \chi\bigr)_\Omega.
  \end{equation*}
  Moreover, $\tfrac{1}{\ell} (\rhol \lvert \Delta \chi)_\Omega$ tends to $0$
  because $(\rhol)_\ell$ is bounded in $L^2(\Omega)$ and
  $\Delta \chi \in L^2(\Omega)$. So we deduce
  \begin{equation*}
    \bigl(\rhol \nabla \phil \bigm| \nabla \chi\bigr)_\Omega \to 0.
  \end{equation*}
  And so, by the uniqueness of the limit, and with \eqref{eq : cvg
    weak rho phi}, $(\rho \nabla \varphi \lvert \nabla \chi)_\Omega = 0$
  for every $\chi \in \mathcal{D}(\Omega)$. This implies
  $\Div(\rho \nabla \varphi) = 0$ in the sense of distributions.

  Consequently $(\varphi,\rho)$ is a solution of \eqref{sys:sys_sans_A}.

  In order to justify that the solution $(\varphi,\rho)$ is different
  from the trivial solution $(\phie,0)$ defined by \eqref{sys:elec}, it
  remains to show that $\rho \not \equiv 0$. Since
  $0\leq \rhol \leq K_+ = \rhoch$ a.e. in $\Omega$,
  $\partial \rhol /\partial\normal \ge 0$ on $\gammac$.

  Let us fix $\chi \in H^2(\Omega)$ such that $\chi \geq 0$ on
  $\gammac$ and $\chi \equiv 0$ on $\gammad$. So
  \begin{align*}
    0 &= \int_\Omega \Div \bigl(\tfrac{1}{\ell} \nabla \rhol
        + \rhol \nabla \phil \bigr) \chi \intd x\\
      & = - \int_\Omega
        \bigl(\tfrac{1}{\ell} \nabla \rhol + \rhol \nabla \phil\bigr)
        \nabla \chi \intd x
        + \int_{\gammac} \bigl(\tfrac{1}{\ell} \nabla \rhol \cdot \normal
        + \rhol \nabla \phil \cdot \normal\bigr) \chi \intd x \\
      & = - \int_\Omega \rhol \nabla \phil \nabla \chi \intd x
        + \frac{1}{\ell} \int_\Omega \rhol \Delta \chi \intd x
        - \frac{1}{\ell} \int_{\partial \Omega} \rhol
        \frac{\partial \chi}{\partial\normal} \intd x
        + \int_{\gammac} \Bigl(\frac{1}{\ell}
        \frac{\partial \rhol}{\partial\normal}
        + \rhoch \frac{\partial \phil}{\partial\normal}\Bigr) \chi \intd x.
  \end{align*}
  Using ${\partial \rhol}/{\partial\normal} \ge 0$ on $\gammac$, this
  implies
  \begin{equation}
    \label{eq : int rho}
    \int_\Omega \rhol \nabla \phil \nabla \chi \intd x
    \ge \frac{1}{\ell} \int_\Omega \rhol \Delta \chi \intd x
    - \frac{1}{\ell} \int_{\partial\Omega} \rhol
    \frac{\partial \chi}{\partial\normal} \intd x
    +  \int_{\gammac} \rhoch \frac{\partial \phil}{\partial\normal}
    \chi \intd x.
  \end{equation}
  We want to pass to the limit $\ell\to\infty$ in this inequality.
  First, since
  $\int_\Omega \rhol \Delta \chi \intd x$ tends to
  $\int_\Omega \rho \Delta \chi \intd x$, one has
  \begin{equation*}
    \frac{1}{\ell} \int_\Omega \rhol \Delta \chi \intd x
    \xrightarrow[\ell\to\infty]{} 0.
  \end{equation*}

  Secondly, since
  $$\int_{\partial \Omega} \rhol \frac{\partial
    \chi}{\partial\normal} \intd x = \int_{\gammac} \rhoch \frac{\partial
    \chi}{\partial\normal} \intd x + \int_{\gammad} \rhod \frac{\partial
    \chi}{\partial\normal} \intd x$$ is a constant, this implies that
  \begin{equation*}
    \frac{1}{\ell}
    \int_{\partial \Omega} \rhol \frac{\partial \chi}{\partial\normal} \intd x
    \xrightarrow[\ell\to\infty]{} 0.
  \end{equation*}
  Moreover, with \eqref{eq : cvg weak rho phi} and since
  $\partial \phil/\partial\normal$ tends to $\partial \varphi/\partial\normal$
  strongly in $L^2(\gammac)$, we can pass to the limit in \eqref{eq : int rho} and we obtain:
  \begin{equation*}
    \int_\Omega \rho \nabla \varphi \nabla \chi \intd x
    \geq \int_{\gammac} \rhoch \frac{\partial \varphi}{\partial\normal} \chi \intd x.
  \end{equation*}
  This implies that $\rho \not \equiv 0$. Indeed, if $\rho \equiv 0$,
  then $\varphi = \phie$ and so we will have
  \begin{equation*}
    0 \geq \int_{\gammac} \rhoch \frac{\partial \phie}{\partial\normal} \chi \intd x.
  \end{equation*}
  But $\rhoch >0$ and $\partial \phie / \partial\normal \geq 0$ is non
  identically zero (thanks to the strong maximum principle),
  choosing $\chi >0$ on $\gammac$ yields the contradiction
  \begin{equation*}
    0 \ge \int_{\gammac} \rhoch \frac{\partial \phie}{\partial\normal}
    \chi \intd x >0.
  \end{equation*}
  So $\rho \not \equiv 0$ and we have proved the
  existence of a nontrivial solution $(\varphi,\rho)$ to the
  system~\eqref{sys:sys_sans_A}.

  To conclude, let us show the existence of infinitely many solutions.
  Since $\rho \not \equiv 0$, there exists a positive constant
  $K_0 \le \rhoch$ such that $\|\rho\|_{L^\infty(\Omega)} = K_0$. If
  we repeat the above argument with $\rho = K_0/2$ on $\gammac$ and
  $\rho=0$ on $\gammad$, we obtain the existence of a nontrivial
  solution $(\varphi_1,\rho_1)$ to Problem~\eqref{sys:sys_sans_A}.
  The bound $\|\rho_1\|_{L^\infty(\Omega)} \le K_0/2$ implies that
  $\rho_1 \ne \rho$.  We conclude by iterating the above procedure.
\end{proof}

\section{Properties on the normal derivative of the solution \label{s:properties}}

In this section, we establish bounds on the normal derivative of
solutions to~\eqref{sys:gen} in term of the normal derivatives of the
electrostatic solution $\phie$.  We also provide quantitative
estimates on the latter in the special case when $\gammac$ is a
circle.

Let us start with a theorem that provides necessary and sufficient conditions on
$A = \partial \varphi/\partial\normal |_{\gammac}$ for a solution of
Poisson's equation with a nonnegative right-hand-side to exist.

\begin{theorem}\label{thm : deriv_normal}
  Let $\Omega$ be an open subset of $\IR^2$ satisfying the assumptions of the introduction and
  $\varphi \in H^1(\Omega)$ be the unique solution of
  \begin{equation}\label{sys:phi_laplace}
    \begin{cases}
      -\Delta \varphi = \rho, &  \text{in } \Omega, \\
      \varphi = 0, &  \text{on } \gammad, \\
      \varphi =1, &  \text{on } \gammac,
    \end{cases}
  \end{equation}
  where
  $\rho \in L^p(\Omega)$, with $p>2$ and $\rho \geq 0$. Let
  $\phie \in H^1(\Omega)$ be the solution of \eqref{sys:elec}. Then we have 
  \begin{equation*}
    \frac{\partial \phie}{\partial\normal} > 0, 
    \text{ on } \gammac \qquad \text{and} \qquad   \frac{\partial \phie}{\partial\normal} < 0,
    \text{ on } \gammad \setminus \mathcal{V},
  \end{equation*}
  where $\mathcal{V}$ is the set of corners of $\gammad$.
   Moreover, the three following properties are equivalent:
  \begin{align}
    \rho &\not \equiv 0, \label{eq: rho equiv 0}\\
    \frac{\partial \varphi}{\partial\normal}
    &< \frac{\partial \phie}{\partial\normal},
    \quad \text{on } \gammac,
    \label{eq : ineg_deriv_normal}\\[1\jot]
    \frac{\partial \varphi}{\partial\normal}
    &< \frac{\partial \phie}{\partial\normal},
    \quad \text{on } \gammad \setminus \mathcal{V}.
    \label{eq : ineg_deriv_normal_2}
  \end{align}
  
\end{theorem}

\begin{proof}
  First, by using the strong maximum principle and the Hopf boundary point lemma to $\phie$ \cite[Theorem 2.2 \& Lemma 3.4]{gilbarg2001elliptic} (see also \cite[Theorem 3.27, Lemma 3.26]{tro}), we
  obtain that $0<\phie< 1$ and $\partial \phie/\partial\normal < 0$
  on $\gammad \setminus \mathcal{V}$ and $\partial \phie/\partial\normal> 0$ on $\gammac$.

  First, assume that \eqref{eq: rho equiv 0} and prove \eqref{eq : ineg_deriv_normal} and \eqref{eq : ineg_deriv_normal_2}.
  Let us consider  $d \coloneq \varphi-\phie$ which solves
  \begin{equation*}
    \begin{cases}
      - \Delta d = \rho, & \text{in } \Omega, \\
      d = 0, &  \text{on } \partial \Omega.
    \end{cases}
  \end{equation*}
  As $d$ cannot be $0$, thanks to the strong maximum principle, we have
        \begin{equation*}
    d>0 \text{ in } \Omega.
  \end{equation*}
    Thus by the Hopf boundary point lemma, we have $\partial d / \partial\normal <0$ on $\partial \Omega \setminus \mathcal{V}$, which establishes \eqref{eq : ineg_deriv_normal} and \eqref{eq : ineg_deriv_normal_2}.
  
  Now let us prove that \eqref{eq : ineg_deriv_normal} (resp.~\eqref{eq : ineg_deriv_normal_2}) implies \eqref{eq: rho equiv 0}. Assume on the contrary that $\rho \equiv 0$. Then, $\varphi = \phie$ contradicting \eqref{eq : ineg_deriv_normal} (resp.~\eqref{eq : ineg_deriv_normal_2}).
\end{proof}

Let us now turn to the quantitative estimates on
$\partial\phie/\partial\normal$.

\begin{theorem}
  \label{thm : phie_1_2}
  Assume that $\gammac=\partial B(0,r_0) \subseteq \IR^2$, for some
  $r_0>0$ and let $\phie$ be the solution of \eqref{sys:elec}. Then
  \begin{equation}\label{ineq : phie_radial}
    \frac{1}{r_0 \ln(r_2/r_0)}
    \le \frac{\partial \phie}{\partial\normal}\biggr|_{\gammac}
    \le \frac{1}{r_0 \ln(r_1/r_0)},
  \end{equation}
  where $r_1 = \min_{x \in \gammad} |x|$ and
  $r_2 = \max_{x \in \gammad} |x|$.
\end{theorem}

Note that $r_2 \ge r_1 > r_0$ are such that
$B(0,r_1) \subseteq \Omega \subseteq B(0,r_2)$, see
Figure~\ref{fig : domaine derivée} for an illustration.

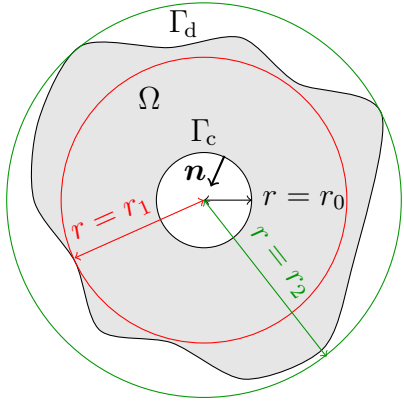
\begin{figure}[ht]
  \centering
  \begin{tikzpicture}[baseline = 1ex, scale = 0.9]
    \path
    \foreach \X/\Y in {1/0.9,2/0.2,3/0.45,4/0.3,5/0.9,6/0.7,7/0.5,8/0.1,
      9/0.45,10/0.15,11/0.7,12/0.8,13/0.4}{
      (360*\X/14:2+\Y) coordinate (A\X)
    };
    \draw[fill=black!10] plot[smooth cycle,samples at={1,...,13}] (A\x);
    \draw[fill=white] (0,0) circle (0.7);
    \draw[draw=red] (0,0) circle (2.1);
    \draw[draw=red, <->] (0,0) -- (-1.9,-0.85) node[red,pos=0.65, sloped, above]{$r=r_1$};
    \draw[<->] (0,0) -- (0.7,0) node[right]{$r=r_0$};
    \node at (0,1) {$\gammac$};
    \node at (-0.3,2.6) {$\gammad$};
    \node at (-0.8,1.5) {$\Omega$};
    \draw[->, thick] (0.3,0.65) -- (0.1,0.2);
    \node[left] at (0.2,0.4) {$\normal$};
    \draw[draw=darkgreen] (0,0) circle (2.9);
    \draw[draw=darkgreen,<->] (0,0) -- (1.8,-2.3) node[darkgreen,pos=0.48, sloped, above]{$r=r_2$};
  \end{tikzpicture}
  \caption{A domain with a circular conductor with the inscribed
      and circumscribed disks.}
  \label{fig : domaine derivée}
\end{figure}

\begin{proof}
  For $i \in \{1, 2\}$, let $\phie^{(i)}$ be given by
  \begin{equation} \label{eq : phiei def}
    \phie^{(i)}(x) \coloneq \frac{\ln(r/r_i)}{\ln(r_0/r_i)}.
  \end{equation}
  We readily check that $\phie^{(i)}$ is the unique solution of
  \begin{equation} \label{sys:phie1}
    \begin{cases}
      \Delta \phie^{(i)} = 0,
      & \text{in }   B(0,r_i)\backslash B(0,r_0), \\
      \phie^{(i)} = 1, & \text{on } \gammac, \\
      \phie^{(i)} = 0, & \text{on } \partial B(0,r_i).
    \end{cases}
  \end{equation}

  Let us start with the right inequality of \eqref{ineq : phie_radial}. Thanks to \eqref{eq : phiei
    def}, we directly see that $\phie^{(1)} \leq 0$ on $\gammad$.
  Then $d^{(1)} \coloneq \phie- \phie^{(1)}$ is a solution to
  \begin{equation*}
    \begin{cases}
      \Delta d^{(1)} = 0, &   \text{in } \Omega , \\
      d^{(1)} = 0, &  \text{on } \gammac, \\
      d^{(1)} \ge 0, & \text{on } \gammad,
    \end{cases}
  \end{equation*}
  and, thanks to the maximum principle
  \cite[Theorem~3.1]{gilbarg2001elliptic}, $d^{(1)}\geq 0$
  in $\Omega$.  Therefore
  \begin{align*}
    \frac{\partial d^{(1)}}{\partial\normal}\biggr|_{\gammac} \le 0
    \qquad\text{i.e.}\qquad
    \frac{\partial \phie}{\partial\normal}\biggr|_{\gammac}
    \le \frac{\partial \phie^{(1)}}{\partial\normal}\biggr|_{\gammac}
    = \frac{-1}{r_0 \ln(r_0/r_1)}.
  \end{align*}
  The same argument implies the left inequality of \eqref{ineq : phie_radial} because
  \eqref{eq : phiei def} yields $\phie^{(2)} \ge 0$ on~$\gammad$.
\end{proof}

\begin{remark}
  Assuming that $\gammac=\partial B(0,r_0)$, for some $r_0>0$,
  Theorems \ref{thm : deriv_normal} and \ref{thm : phie_1_2} allows us
  to draw the following two conclusions.

  If
  \begin{equation*}
    A = \dfrac{\partial \varphi}{\partial\normal}
    > \frac{1}{r_0 \ln(r_1/r_0)}
    \text{ on } \gammac,
  \end{equation*}
  then \eqref{sys:phi_laplace} has no solution and neither has
  Problem~\eqref{sys:gen}--\eqref{sys:BC}.

   If
   \begin{equation*}
     A = \dfrac{\partial \varphi}{\partial\normal}
     \le \frac{1}{r_0 \ln(r_2/r_0)}
          \text{ on } \gammac,  
   \end{equation*}
   then a solution to Problem~\eqref{sys:gen}--\eqref{sys:BC}
   \emph{may} exist.
\end{remark}

\section{Analytical solutions for the unipolar radial case}
\label{s:radial}

In this section, we consider a simpler domain $\Omega$. More precisely
$\Omega$ is supposed to be an annulus (see Figure~\ref{fig : radial}) and
is given by
$\Omega \coloneq \bigl\{ (r \cos\theta, r\sin\theta)\in \IR^2
\bigm| r_0<r<1,\linebreak[2]\,
0\leq \theta\leq 2 \pi \bigr\}$ where $0<r_0<1$.
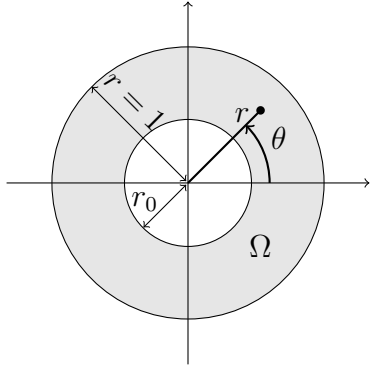
\begin{figure}[h!t]
  \centering
  \begin{tikzpicture}[baseline = 1ex, scale = 1.2]
    \draw[fill=black!10] (0,0) circle (1.5);
    \draw[fill=white] (0,0) circle (0.7);
    \draw[->] (-2,0) -- (2,0);
    \draw[->] (0,-2) -- (0,2);
    \draw[<->] (-0.02, 0.02) -- (-1.06, 1.06)
    node[pos=0.7, sloped, above] {$r=1$};
    \draw[<->] (-0.02, -0.02) -- (-0.49,-0.49)
     node[pos=0.4, left]{$r_0$};
    \draw[thick, fill] (0,0) -- (0.8, 0.8) circle(1pt)
    node at (0.6, 0.75) {$r$};
    \draw[thick,->] (0.9,0) arc (0:45:0.9) node at (1,0.5) {$\theta$};
    \node at (0.8, -0.7){$\Omega$};
  \end{tikzpicture}
  \caption{Definition domain in the radial case.}
  \label{fig : radial}
\end{figure}

Here $\gammac = \partial B(0,r_0)$ and $\gammad = \partial B(0,1)$.
Our aim is to compute all the radial solutions to
\eqref{sys:gen}--\eqref{sys:BC}, that is all radial $(\varphi, \rho)$
satisfying
\begin{subnumcases}{}
  -\Delta \varphi(r) = \rho(r), \label{eq : laplace_radial}\\
  \Div \bigl(\rho(r) \nabla\varphi(r)\bigr) = 0,\label{eq : div_radial}\\
  \varphi(r=r_0) = 1, \label{eq : cb_r_0}\\
  \varphi(r=1) = 0, \label{eq : cb_r_1}\\[2\jot]
  \displaystyle \frac{\partial \varphi}{\partial r}(r=r_0)
  = -A, \label{eq : cb_neumann_radial}
\end{subnumcases}
where $A$ is a strictly positive constant.
These solutions are given by the next theorem.

\begin{theorem}
  \label{thm : sol_radial}
  All but one radial solutions $(\varphi,\rho)$ of the system
  \eqref{eq : laplace_radial}--\eqref{eq : cb_r_1} are given by:
  \begin{equation} \label{eq : sol_radial_phi}
    \varphi_\lambda(r)
    = \frac{F_\lambda(1)-F_\lambda(r)}{F_\lambda(1)-F_\lambda(r_0)},
  \end{equation}
  and
  \begin{equation}\label{eq : sol_radial_rho}
    \rho_\lambda(r) = \begin{cases}
      \displaystyle \frac{\lambda}{(F_\lambda(1)-F_\lambda(r_0))
      \sqrt{1+\lambda r^2}} & \text{if } \lambda \geq -1,\\[4\jot]

      \displaystyle\frac{-\lambda}{(F_\lambda(1)-F_\lambda(r_0))
      \sqrt{-1-\lambda r^2}} & \text{if } \lambda
                               \leq \displaystyle \frac{-1}{r_0^2}<-1,\\[4\jot]
    \end{cases}
  \end{equation}
  for $r_0 \leq r \leq 1$,
  where $\lambda$ is a real parameter varying in
  $\Bigintervaloc{-\infty, \displaystyle \frac{-1}{r_0^2}}\cup
    \intervalco{-1, + \infty}$. The function
  $F_\lambda: \intervalcc{r_0, 1} \to \IR$ is
  increasing and defined as
  follows:
  \begin{equation}
    \label{eq:def-F}
    F_{\lambda}(r) =
    \begin{cases}
      \sqrt{1+\lambda r^2} - \ln(\sqrt{1+\lambda r^2}+1) + \ln(r)
      & \text{if } \lambda \geq -1,\\[1\jot]
      \sqrt{-\lambda r^2 -1} - \arctan(\sqrt{-\lambda r^2-1})
      & \text{if } \lambda\leq \displaystyle \frac{-1}{r_0^2}. \\[2\jot]
    \end{cases}
  \end{equation}
  The Neumann condition \eqref{eq : cb_neumann_radial} is satisfied
  for the following value of $A$:
  \begin{equation}\label{lemma : A_radial}
        A_\lambda
    \coloneq
    \frac{\sqrt{\strut \abs{\lambda + r_0^{-2}}}}{
      F_\lambda(1) -F_\lambda(r_0)}.
  \end{equation}
  The remaining radial solution is:
  \begin{equation*}
    \varphi_\infty(r) \coloneq \frac{1-r}{1-r_0},\quad
    \rho_\infty(r) \coloneq \frac{1}{r(1-r_0)}, \quad
    \text{with}\quad
    A_\infty \coloneq \frac{1}{1-r_0}.
  \end{equation*}

  Finally, for each $A \in \intervalcc{0, A_{-1}}$, there is a unique
  solution $(\varphi, \rho)$ to \eqref{eq :
    laplace_radial}--\eqref{eq : cb_neumann_radial} and the maps
  \begin{align*}
    &\intervalcc{0, A_{-1}} \to H^1\bigl(\intervaloo{r_0,1}\bigr):
    A \mapsto \varphi_\lambda
    \text{ with } A_\lambda = A,\\
    &\intervaloo{0, A_{-1}} \to L^2\bigl(\intervaloo{r_0,1}\bigr):
    A \mapsto \rho_\lambda
    \text{ with } A_\lambda = A
  \end{align*}
  are continuous.
\end{theorem}

In the previous Theorem, the electrostatic solution corresponds to
$\lambda=0$ and is given by:
\begin{equation}\label{eq : phie_radiale}
  \varphi_0(r) = \phie(r) = \frac{\ln(r)}{\ln(r_0)}.
\end{equation}
In that case, the Neumann condition holds for $A$ being
$A_0 = {-1}/({r_0 \ln(r_0)})$.

In Figure \ref{fig : cond_neumann}, the graph of the function
$\lambda \mapsto A_\lambda$ is drawn which allows to visualize the
relationship between $\lambda$ and $A$.

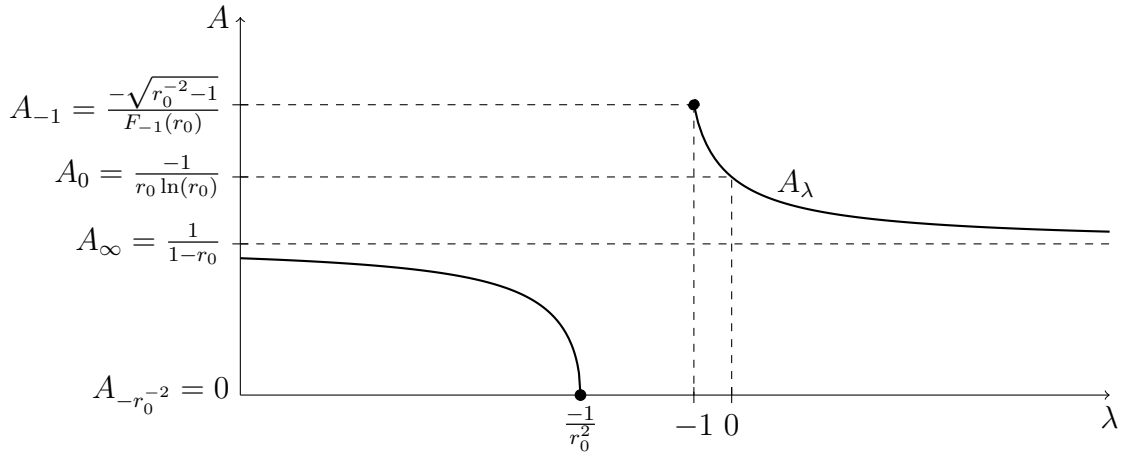
\begin{figure}[h!t]
  \centering
  \newcommand{\lmin}{-13}
  \begin{tikzpicture}[x=5mm]
    \draw[->] (\lmin,0) -- (10,0) node[below]{$\lambda$};
    \draw[->] (\lmin, 0) -- ++(0, 5) node[left]{$A$};
    \draw[thick] plot file {\data{A_lambda1.dat}};
    \draw[thick] plot file {\data{A_lambda2.dat}};
    \node at (1.7, 2.8) {$A_\lambda$};
    \draw (\lmin,2) +(-3pt, 0) node[left]{$A_{\infty} = \frac{1}{1-r_0}$}
    -- +(3pt, 0);
    \draw[dashed] (\lmin, 2) -- (10, 2);
    \draw[fill] (-4,0) circle(2pt) node[below] {$\frac{-1}{r_0^2}$};
    \draw (\lmin,3.84) +(-3pt, 0) node[left]{      $A_{-1} = \frac{-\sqrt{r_0^{-2}-1}}{F_{-1}(r_0)}$} -- +(3pt, 0);
    \draw[dashed, fill] (\lmin, 3.8410) -- (-1, 3.8410) circle(2pt)
    -- (-1, 0);
    \draw (-1, 3pt) -- (-1, -3pt) node[below]{$-1$};
    \draw (\lmin, 2.8854) +(-3pt, 0) node[left]{      $A_0 = \frac{-1}{r_0 \ln(r_0)}$} -- +(3pt, 0);
    \draw[dashed] (\lmin, 2.8854) -- (0, 2.8854) -- (0,0);
    \draw (0, 3pt) -- (0, -3pt) node[below]{$0$};
    \node[left] at (\lmin, 0) {$A_{-r_0^{-2}} = 0$};
  \end{tikzpicture}
  \caption{Graph of the function $\lambda \mapsto A_\lambda$.}
  \label{fig : cond_neumann}
\end{figure}

\begin{proof}[Proof of Theorem \ref{thm : sol_radial}]
  Let us start by computing all radial solutions to \eqref{eq :
    laplace_radial}--\eqref{eq : cb_neumann_radial}.  Because
  $\varphi$ and $\rho$ are radial functions, Equations~\eqref{eq :
    laplace_radial} and \eqref{eq : div_radial} become, respectively,
  \begin{gather}
    \label{eq: eq_laplacien_radial}
    -\frac{\partial^2 \varphi}{\partial r^2}(r)
    - \frac{1}{r} \frac{\partial \varphi}{\partial r}(r)  = \rho(r),
    \\[1\jot]
    \frac{\partial}{\partial  r}
    \Bigl(\rho(r) \frac{\partial \varphi}{\partial r}(r)\Bigr)
    + \frac{1}{r} \rho(r) \frac{\partial \varphi}{\partial r}(r)
    = 0.
  \end{gather}
  The second equation is equivalent to the fact that there exists a
  real constant $K$ such that
  \begin{equation}\label{eq : def_rho}
    \forall r \in \intervaloo{r_0, 1},\qquad
    \rho(r) \frac{\partial \varphi}{\partial r}(r) = \frac{K}{r}.
  \end{equation}
  Multiplying \eqref{eq: eq_laplacien_radial} by
  $\partial\varphi/\partial r$ and using \eqref{eq : def_rho} yields
  \begin{equation}\label{eq : dphi 0}
    -\frac{\partial^2 \varphi}{\partial r^2}(r)
    \frac{\partial \varphi}{\partial r}(r) - \frac{1}{r}
    \Bigl(\frac{\partial \varphi}{\partial r}(r)\Bigr)^2
    = \frac{K}{r}.
  \end{equation}
  This is equivalent to \eqref{eq: eq_laplacien_radial}
  provided that $(\partial \varphi / \partial r)(r) \ne 0$ for almost every
  $r\in \intervaloo{r_0, 1}$.  Equation \eqref{eq : dphi 0} can be
  seen as a linear equation in $(\partial\varphi/\partial r)^2$ and solving it
  implies that
  \begin{equation}\label{eq : dphi carré}
    \forall r\in \intervaloo{r_0, 1},\qquad
    \Bigl(\frac{\partial \varphi}{\partial r}(r)\Bigr)^2
    = \frac{\longtilde{K}}{r^2} -K
  \end{equation}
  for some constant $\longtilde{K} \in \IR$.  It is not possible that
  $K$ and $\longtilde{K}$ both vanish as that would imply that $\varphi$
  is constant and so cannot satisfy the boundary conditions
  \eqref{eq : cb_r_0} and \eqref{eq : cb_r_1}.  Therefore $\partial\varphi/\partial r$ vanishes at at most one point and so \eqref{eq : dphi
    carré} is equivalent to \eqref{eq: eq_laplacien_radial}.
    Depending on the value of $\longtilde{K}$, we can distinguish two
  cases.
  \begin{enumerate}[label=\arabic*)]
  \item \textbf{Case $\longtilde{K}=0$.}

    Equality \eqref{eq : dphi carré} says that $\partial\varphi/\partial r$ is
    constant and the boundary condition \eqref{eq : cb_neumann_radial}
    that $\partial\varphi/\partial r \equiv -A$ and $A^2 = -K$.  Hence, using
    \eqref{eq : cb_r_1},
    \begin{equation*}
      \varphi(r) = \int_r^1  A \intd s =  A  (1-r).
    \end{equation*}
    The boundary condition \eqref{eq : cb_r_0} leads to
    \begin{equation*}
      A = A_\infty \coloneq \frac{1}{1-r_0},
    \end{equation*}
    and so $ \varphi(r)
    =\frac{1-r}{1-r_0}$.   Now using \eqref{eq : def_rho}
    and $K = -A^2$, we get $\rho(r) = \frac{1}{r(1-r_0)}$.

  \item \textbf{Case $\longtilde{K}\ne 0$.} 
    Let us set $\lambda \coloneq -K/\longtilde{K}$.  Equality
    \eqref{eq : dphi carré} thus becomes
    \begin{equation} \label{ineq : cas lambda}
      \forall r\in \intervaloo{r_0, 1},\qquad
      \Bigl(\frac{\partial \varphi}{\partial r}(r)\Bigr)^2
      = \longtilde{K} (r^{-2}+\lambda).
    \end{equation}
    If the right hand side vanishes in $\intervaloo{r_0, 1}$, it
    changes sign and so $\partial\varphi/\partial r$ will not be defined on the
    whole $\intervaloo{r_0, 1}$.  Therefore $\longtilde{K}
    (r^{-2}+\lambda) > 0$ for all $r \in \intervaloo{r_0,1}$ or,
    equivalently, $\lambda \ge -1$ if $\longtilde{K} > 0$ and
    $\lambda \ge -r_0^{-2}$ if $\longtilde{K} < 0$.
    Given the boundary conditions \eqref{eq : cb_r_0} and \eqref{eq :
      cb_r_1}, $\partial\varphi/\partial r$ must be negative at some $r \in
    \intervaloo{r_0, 1}$, so
    \begin{equation}
      \label{eq:dphi}
      \frac{\partial\varphi}{\partial r}(r)
      = - \sqrt{\longtilde{K} (r^{-2}+\lambda)}
    \end{equation}
    and, using \eqref{eq : cb_r_1}, we get
    \begin{equation*}
      \varphi(r)
      = \int_r^1 \sqrt{\longtilde{K}(s^{-2} + \lambda)} \intd s.
    \end{equation*}
    Distinguishing the cases $\longtilde{K} > 0$ and
    $\longtilde{K} < 0$, we find after integration that
    \begin{equation}\label{eq : def_phi_F}
      \varphi(r) = \sqrt{\abs{\longtilde{K}}}
      \bigl(F_\lambda(1)-F_\lambda(r)\bigr)
    \end{equation}
    where $F_\lambda$ is defined in the statement of the theorem.
    Imposing the boundary condition \eqref{eq : cb_r_0} enables to
    determine $\sqrt{\abs{\longtilde{K}}}$ and then to rewrite
    \eqref{eq : def_phi_F} as
    \begin{equation}\label{eq : phi_lambda}
      \varphi_\lambda(r)
      = \frac{F_\lambda(1)-F_\lambda(r)}{F_\lambda(1)-F_\lambda(r_0)}.
    \end{equation}
    It remains to impose condition~\eqref{eq : cb_neumann_radial}.  In
    view of \eqref{eq:dphi}, it is equivalent to
    $A = \sqrt{\longtilde{K} (r_0^{-2}+\lambda)}$ or, equivalently,
    given the value for $\sqrt{\abs{\longtilde{K}}}$ found above,
    \begin{equation}
      \label{eq : eqA}
      A = A_\lambda
      \coloneq \frac{\sqrt{\abs{\lambda+r_0^{-2}}}}{F_\lambda(1)-F_\lambda(r_0)}.
    \end{equation}
    The expression \eqref{eq : sol_radial_rho} for $\rho$ follows from
    \eqref{eq : def_rho}, \eqref{eq:dphi} and
    $\lambda = -K/\longtilde{K}$.
  \end{enumerate}
  To prove uniqueness of the radial solution, it remains to establish
  that the map $\mathcal{A}: \lambda \mapsto A_\lambda$ is one-to-one
  on
  $\lambda \in \intervaloc{-\infty, -r_0^{-2}} \cup \intervalco{-1,
    +\infty} \cup \{\infty\}$.  To that end, it suffices to prove that
  $\mathcal{A}$ is decreasing on $\intervaloc{-\infty, -r_0^{-2}}$ and
  on $\intervalco{-1, +\infty}$ and that $A_\lambda \to A_\infty$ as
  $\lambda \to \pm\infty$.  Note that, since $\mathcal{A}$ is
  obviously continuous on
  $\intervaloc{-\infty, -r_0^{-2}} \cup \intervalco{-1, +\infty}$ in
  view of \eqref{eq:def-F} and \eqref{eq : eqA}, this also implies
  that the image of $\mathcal{A}$ is
  $\intervalco{A_{-r_0^{-2}}, A_\infty} \cup \intervaloc{A_\infty,
    A_{-1}} \cup \{A_\infty\} = \intervalcc{0, A_{-1}}$.
  Consequently, the inverse of $\mathcal{A}$,
  \begin{equation*}
    \mathcal{A}^{-1}: \intervalcc{0, A_{-1}} \to \Dom\mathcal{A}:
    A \mapsto \lambda \text{ such that } A_\lambda = A
  \end{equation*}
  is continuous when $\Dom\mathcal{A}$ is seen as the three pieces
  $\intervalco{-r_0^{-2}, -\infty}$, $\{\infty\}$, and
  $\intervaloc{+\infty, -1}$ glued together with the topology
  coming from the compactification of $\IR$ with a single point at
  infinity.

  Let now show that $\mathcal{A}$ is decreasing by showing that
  $\partial_\lambda A_\lambda < 0$ on $\intervaloo{-\infty, r_0^{-2}}
  \cup \intervaloo{-1,+\infty}$.
  Thanks to \eqref{eq : eqA},
  \begin{equation*}
    \frac{1}{A_\lambda}
    = \frac{F_\lambda(1) - F_\lambda(r_0)}{ \sqrt{|\lambda + r_0^{-2}|}}
    = \frac{r_0}{\sqrt{1+\lambda r_0^2}} \int_{r_0}^1 \frac{1}{s}
    \sqrt{|1+\lambda s^2|} \intd s = \int_{r_0}^1 \frac{r_0}{s}
    \sqrt{\frac{1+\lambda s^2}{1+\lambda r_0^2}}\intd s.
  \end{equation*}
  A direct computation yields,
  \begin{equation*}
    \partial_\lambda \Biggl(\frac{r_0}{s}
    \sqrt{\frac{1+\lambda s^2}{1+ \lambda r_0^2}}\Biggr)
            = \frac{r_0}{2s} \, \sqrt{\frac{1+\lambda r_0^2}{1 +\lambda s^{2}}}
    \frac{s^2 -r_0^2}{(1+\lambda r_0^2)^2}.
  \end{equation*}
  Since $s^2-r_0^2 >0$ for $s \in \intervaloc{r_0, 1}$, one deduces
  \begin{equation*}
    \partial_\lambda\Bigl(\frac{1}{A_\lambda}\Bigr)
    = - \frac{\partial_\lambda A_\lambda}{A_\lambda^2} >0,
  \end{equation*}
  whence the claim.

  Let us now turn to the limits as $\lambda \to \pm\infty$.  Using
  l'H\^opital's rule and \eqref{eq:def-F}, it is straightforward to
  show that $F_\lambda(r)/\sqrt{\lambda} \to r$ as $\lambda \to
  +\infty$ and $F_{\lambda}(r) / \sqrt{-\lambda} \to r$ as $\lambda
  \to -\infty$.  The fact that $A_\lambda \to A_\infty$ as $\lambda
  \to \pm\infty$ then readily follows from~\eqref{eq : eqA}.

  Finally, it remains to prove that the maps
  \begin{align*}
    &\Dom\mathcal{A} \to H^1(\intervaloo{r_0,1}):
      \lambda \mapsto \varphi_\lambda,
    &\Dom\mathcal{A} \setminus \Bigl\{\frac{-1}{r_0^{2}}, -1\Bigr\}
      \to L^2(\intervaloo{r_0,1}):
      \lambda \mapsto \rho_\lambda
  \end{align*}
  (with the above described topology on
  $\Dom\mathcal{A}$) are continuous.  This results from the explicit
  expressions \eqref{eq:dphi} and \eqref{eq : phi_lambda} for
  $\varphi_\lambda$ and \eqref{eq : sol_radial_rho} for
  $\rho_\lambda$.  When $\lambda \to
  \pm\infty$, we divide both the numerator and denominator by
  $\sqrt{\abs{\lambda}}$.
\end{proof}

Since we have fixed $\varphi=1$ on $\gammac$, a solution $(\varphi,\rho)$ is said to be \emph{physical} if $\rho \ge 0$.
In view of \eqref{eq : sol_radial_rho},
$(\varphi_\lambda,\rho_\lambda)$ is physical if and only if
$\lambda \in \intervaloc{-\infty, -r_0^{-2}} \cup
\intervalco{0,+\infty} \cup \{\infty\}$.  Figure~\ref{fig :
  radial_sol} illustrates these results for $r_0=0.25$.  On the
contrary, for $\lambda \in \intervalco{-1, 0}$, the solutions are
nonphysical because $\rho_\lambda < 0$ (see Figure~\ref{fig :
  radial_sol_non_physics}).
Note that, when $\lambda = -1/r_0^2$ (resp.\ $\lambda = -1$),
$\rho_\lambda$ is singular at $r = r_0$ (resp.\ $r = 1$) and does not
belong to $L^2(\intervaloo{r_0, 1})$; see Figures~\ref{fig : rho}
and~\ref{fig : rho_non_physics} for an illustration.

\begin{figure}[h!t]
  \begin{subfigure}[b]{0.49\linewidth}
    \centering
    \begin{tikzpicture}[x = 7cm, y=62mm]
      \draw[->] (0.2,0) -- (1,0) -- +(3ex,0) node[below]{$r$};
      \draw[->] (0.2, 0) -- (0.2, 1) -- +(0, 3ex) node[left]{$\varphi$};
      \draw[thick, color = pink!90!black]
      plot file {\data{phi_lambda_0.dat}};
      \draw[thick, color = purple]
      plot file {\data{phi_lambda_4.dat}};
      \draw[thick, color =  brown!70!black]
      plot file {\data{phi_lambda_100.dat}};
      \draw[thick, color =  green!50!black]
      plot file {\data{phi_lambda_m16.dat}};
      \draw[thick, color = orange]
      plot file {\data{phi_lambda_m25.dat}};
      \draw[thick, color = blue]
      plot file {\data{phi_lambda_m100.dat}};

        \foreach \x in {0.25,0.4,0.6,0.8,1}{
        \draw (\x, 1pt) -- (\x, -1pt) node[below, scale=0.6]{$\x$};
      }
      \foreach \y in {0,0.5,1}{
        \draw (0.2, \y) ++(3pt,0) -- ++(-6pt,0) node[left, scale=0.6]{$\y$};
      }
      \begin{scope}[shift={(55mm, -5mm)}]
        \foreach \c/\dy/\l in {
          pink!90!black/3/0, purple/4/4, brown!70!black/5/100,
          green!50!black/2/-16, orange/1/-25,blue/0/-100
        }{
          \draw[color=\c, ultra thick]
          (0, 0.7) ++(0, \dy\baselineskip) -- ++(1em,0)
          node[black, right] {\footnotesize $\lambda = \l$};          
        }
      \end{scope}
                            \end{tikzpicture}
    \caption{Electrical potential $\varphi$.}
    \label{fig : phi}
  \end{subfigure}   \begin{subfigure}[b]{0.49\linewidth}
    \centering
    \begin{tikzpicture}[x=7cm, y=3ex]
      \draw[->] (0.2, 0) -- (1, 0) -- +(3ex, 0) node[below]{$r$};
      \draw[dotted,yscale=0.8] (0.2, 5.5) -- +(0, 5ex);
      \draw[yscale=0.8] (0.2, 0) -- (0.2, 5.5) -- +(0, 1ex);
      \begin{scope}
        \clip (0.2, 0) rectangle (1, 4.5);
        \foreach \la/\c in {
          4/purple, 100/{brown!70!black},
          m16/green!50!black, m25/orange,
          m100/blue}
        {
          \draw[thick, color=\c,yscale=0.8]
          plot file {\data{rho_lambda_\la.dat}};
        }
        \draw[ultra thick, color =pink!90!black, yscale=0.8] plot file {\data{rho_lambda_0.dat}};
      \end{scope}
      \foreach \x in {0.25,0.4,0.6,0.8,1}{
        \draw (\x, 3pt) -- (\x, -3pt)
        node[below]{$\scriptstyle\x$};
      }
      \foreach \y in {0,...,5}{
        \draw[yscale = 0.8] (0.2, \y) +(3pt, 0) -- +(-3pt, 0)
        node[left]{$\scriptstyle \y$};
      }
      \draw[dotted, color=orange] (0.326, 4.55) -- (0.29, 5.85);
      \draw[dotted, color=green!50!black]
      (0.373, 4.54) -- (0.337, 5.82);
      \begin{scope}[y=1.8mm, yshift=12ex]
        \draw[->,yscale = 0.8] (0.2, 6) -- (0.2, 32) node[left]{$\rho$};
        \foreach \y in {7, 13, 19, 25, 29}{
          \draw[yscale = 0.8] (0.2, \y) +(3pt, 0) -- +(-3pt, 0)
          node[left]{$\scriptstyle\y$};
        }
        \clip (0.2, 5.5) rectangle (1, 30);
        \foreach \la/\c in {m16/green!50!black, m25/orange}{
          \draw[thick, color=\c,yscale = 0.8]
          plot file {\data{rho_lambda_\la.dat}};
          \draw[dotted,green!50!black,yscale=0.8] (0.256,28.91) -- (0.253,32);
        }
      \end{scope}
     \begin{scope}[shift={(55mm, 33mm)}]
        \foreach \c/\dy/\l in {
          pink!90!black/3/0, purple/4/4, brown!70!black/5/100,
          green!50!black/2/-16, orange/1/-25,blue/0/-100
        }{
          \draw[color=\c, ultra thick]
          (0, 0.7) ++(0, \dy\baselineskip) -- ++(1em,0)
          node[black, right] {\footnotesize $\lambda = \l$};          
        }
      \end{scope}
    \end{tikzpicture}

                    \caption{Charge density $\rho$.}
    \label{fig : rho}
  \end{subfigure}
  \caption{Physical radial solutions for $r_0=0.25$ and different
    values of $\lambda$.}
  \label{fig : radial_sol}
\end{figure}
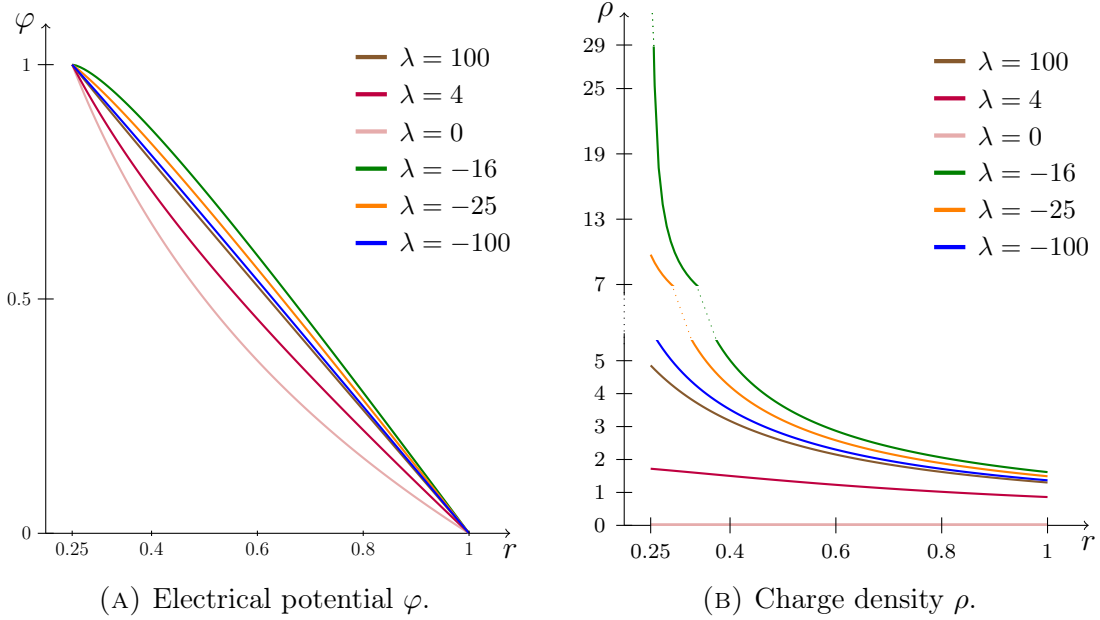

\begin{figure}[h!t]
  \begin{subfigure}[b]{0.49\linewidth}
    \centering
    \begin{tikzpicture}[x = 7cm, y=55mm]
      \draw[->] (0.2,0) -- (1,0) -- +(3ex,0) node[below]{$r$};
      \draw[->] (0.2, 0) -- (0.2, 1) -- +(0, 3ex) node[left]{$\varphi$};

      \draw[thick, color = yellow!90!black]
      plot file {\data{phi_lambda_m1.dat}};
      \draw[thick, color = teal]
      plot file {\data{phi_lambda_m05.dat}};
      \draw[thick, color = violet]
      plot file {\data{phi_lambda_m075.dat}};
      \draw[thick, color = blue!50!white]
      plot file {\data{phi_lambda_m025.dat}};
      \draw[thick, color = gray!80!black]
      plot file {\data{phi_lambda_m095.dat}};

      \foreach \x in {0.25,0.4,0.6,0.8,1}{
        \draw (\x, 1pt) -- (\x, -1pt) node[below, scale=0.6]{$\x$};
      }
      \foreach \y in {0,0.5,1}{
        \draw (0.2, \y) ++(3pt,0) -- ++(-6pt,0) node[left, scale=0.6]{$\y$};
      }

      \begin{scope}[shift={(55mm, -5mm)}]
        \foreach \c/\dy/\l in {
          blue!50!white/4/-0.25, teal/3/-0.5, violet/2/-0.75,
          gray!80!black/1/-0.95, yellow!90!black/0/-1
        }{
          \draw[color=\c, ultra thick]
          (0, 0.7) ++(0, \dy\baselineskip) -- ++(1em,0)
          node[black, right] {\footnotesize $\lambda = \l$};          
        }
      \end{scope}
    \end{tikzpicture}
    \caption{Non-physical solution $\varphi$.}
    \label{fig : phi_non_physics}
  \end{subfigure}   \begin{subfigure}[b]{0.49\linewidth}
    \centering
    \begin{tikzpicture}[x=10cm, y=4ex]
      \draw[->] (0.2, 0) -- (1, 0) -- +(3ex, 0) node[below]{$r$};
      \draw[dotted,yscale=0.8] (0.2, -7) -- +(0, 5ex);
      \draw[yscale=0.8] (0.2, 0) -- (0.2, -5.5) -- +(0, 1ex);
      \begin{scope}
        \clip (0.2, 0) rectangle (1, -4.5);
        \foreach \la/\c in {
          m1/yellow!90!black, m05/teal,
          m075/violet, m025/blue!50!white,
          m095/gray!80!black}
        {
          \draw[thick, color=\c,yscale=0.8]
          plot file {\data{rho_lambda_\la.dat}};
        }
      \end{scope}
      \foreach \x in {0.25,0.4,0.6,0.8,1}{
        \draw (\x, 3pt) -- (\x, -3pt)
        node[above, pos= 0.4]{$\scriptstyle\x$};
      }
      \foreach \y in {0,...,-5}{
        \draw[yscale = 0.8] (0.2, \y) +(3pt, 0) -- +(-3pt, 0)
        node[left]{$\scriptstyle \y$};
      }
      \draw[dotted, color=yellow!90!black] (0.986, -4.5) -- (0.994, -5.65);
      \begin{scope}[y=1.8mm, yshift=-16ex]
        \draw[->,yscale = 0.8] (0.2, -6.5) -- (0.2, -25) node[left]{$\rho$};
        \foreach \y in {-7, -13, -18, -22}{
          \draw[yscale = 0.8] (0.2, \y) +(3pt, 0) -- +(-3pt, 0)
          node[left]{$\scriptstyle\y$};
        }
        \clip (0.2, -6.5) rectangle (1.1, -24);
        \foreach \la/\c in {m1/yellow!90!black,m095/gray!80!black}{
          \draw[thick, color=\c,yscale = 0.8]
          plot file {\data{rho_lambda_\la.dat}};
          \draw[dotted, yellow!90!black,yscale=0.8] (0.999,-22.27)--(1,-25);
        }
      \end{scope}
      \begin{scope}[shift={(60mm, -65mm)}]
        \foreach \c/\dy/\l in {
          blue!50!white/4/-0.25, teal/3/-0.5, violet/2/-0.75,
          gray!80!black/1/-0.95, yellow!90!black/0/-1
        }{
          \draw[color=\c, ultra thick]
          (0, 0.7) ++(0, \dy\baselineskip) -- ++(1em,0)
          node[black, right] {\footnotesize $\lambda = \l$};          
        }
      \end{scope}
    \end{tikzpicture}

                \caption{Non-physical solution $\rho$.}
    \label{fig : rho_non_physics}
  \end{subfigure}
  \caption{Non-physical radial solution for $r_0=0.25$ and different
    values of $\lambda$}
  \label{fig : radial_sol_non_physics}
\end{figure}
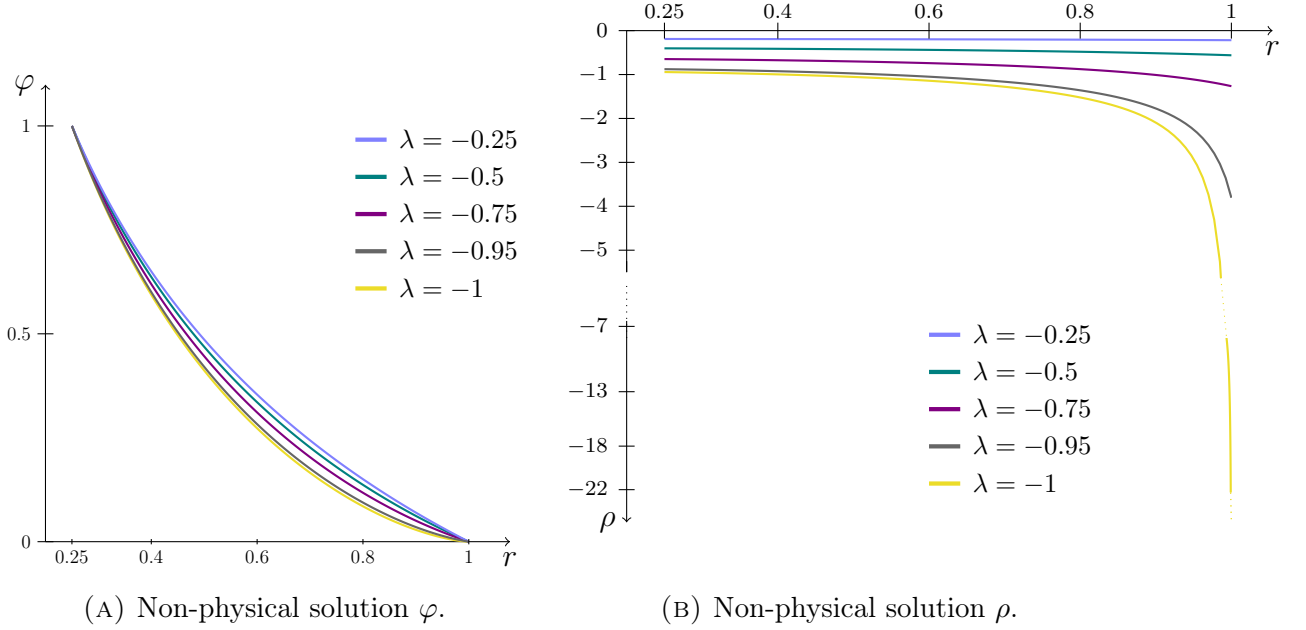

Furthermore, Figures \ref{fig : radial_sol} and \ref{fig :
  radial_sol_non_physics} enable us to visualize the continuous
dependence of solutions $(\varphi,\rho)$ with respect to $A$ which was
mentioned in the statement of Theorem~\ref{thm : sol_radial}.

\begin{remark}
  Solutions of the system \eqref{eq : laplace_radial}--\eqref{eq :
    cb_neumann_radial} in the case $\lambda\ge 0$ were already given
  without proof in \cite[Appendix~A]{Janischewskyj-Cela:1979},
  \cite[formula~(8)]{Sarma-Janischewskyj:1969},
  \cite[Appendix~B]{Takuma-Ikeda-Kawamoto:1981},  
  and
  \cite{Townsend:1914}.  Here, we provide an
  exhaustive determination of (physical) radial solutions and
  demonstrate that the set of solutions as $A$ changes forms a
  continuum.
\end{remark}

\section{A Least-Square Algorithm}
\label{snumerique}

In section \ref{s:unipolar}, we proved the existence of solutions
to~\eqref{sys:sys_sans_A} by adding a small diffusion term,
$\varepsilon \Delta\rho$, and letting $\varepsilon$ tends to zero.
While this approach may be turned into an algorithm, there are two issues.
Firstly, there is no good computational counterpart to the Schauder
fixed point theorem.  The second issue concerns the Neumann boundary
condition in~\eqref{sys:BC}. This condition cannot be controlled
using this approach, even though it is important for the engineers to
be able to specify it. In this section, we present a least squares
approach to solving \eqref{sys:gen}--\eqref{sys:BC} that relies solely
on the classical Finite Element Method (FEM).

Let us start by noticing that we can equivalently write the second
equation of~\eqref{sys:gen} as
$-\Div \bigl((\rho+1) \nabla\varphi \bigr) = \rho$.  The advantage of
this form is that, since $\rho$ can tend to $0$ (see \cite[Section 1]{rolando}),
the ellipticity of the linear operator is kept.
Moreover, the constant of ellipticity is bounded away of $0$ independently
of $\rho$.
In summary, we here
want to compute a numerical approximation of the solution
$(\varphi,\rho)$ to the problem
\begin{equation}\label{sys:diff}
  \begin{cases}
    -\Delta \varphi = \rho, & \text{in } \Omega,\\
    -\Div \bigl((\rho+1) \nabla\varphi\bigr) = \rho, & \text{in } \Omega,\\
    \varphi = 1,&  \text{on } \gammac,\\
    \varphi = 0,&  \text{on } \gammad, \\[2\jot]
    \displaystyle \frac{\partial \varphi}{\partial\normal} = A,
                            & \text{on } \gammac,
  \end{cases}
\end{equation}
where $A$ may be a constant or a sufficiently smooth function.

\subsection{Description of the algorithm}

For a fixed function $\rho$, we divide the system \eqref{sys:diff}
into two subsystems with unknowns $\varphi_1$ and $\varphi_2$ (below
they will be denoted by $\varphi_1(\rho)$ and $\varphi_2(\rho)$ to
emphasize their dependency with respect to $\rho$):
\begin{equation}\label{sys:vf1}
  \begin{cases}
    - \Delta \varphi_1 = \rho, &  \text{in }\Omega,  \\[1\jot]
    \dfrac{\partial \varphi_1}{\partial\normal} = A,
                               & \text{on } \gammac,\\
    \varphi_1  = 0, & \text{on } \gammad,
  \end{cases}
\end{equation}
and
\begin{equation}\label{sys:vf2}
  \begin{cases}
    -\Div\bigl((\rho +1) \nabla \varphi_2 \bigr) = \rho ,
    & \text{in } \Omega,\\
    \varphi_2  = 1,& \text{on } \gammac,\\
    \varphi_2  = 0,& \text{on } \gammad.
  \end{cases}
\end{equation}
Note that both problems have a unique solution in $H^1(\Omega)$ as
soon as $A$ belongs to $H^{1/2}(\gammac)$,
$\rho\in L^\infty(\Omega)$, and $\rho\geq 0$, thanks to Lax-Milgram
theorem.  Then, we introduce the functional
\begin{equation}
  \label{def : J}
  J: L^\infty(\Omega) \rightarrow \IR:
  \rho \mapsto \tfrac{1}{2} \bigl\|\nabla
  \bigl(\varphi_1(\rho)-\varphi_2(\rho)\bigr) \bigr\|^2_\Omega,
\end{equation}
and notice that $J(\rho)$ vanishes if and only if
$\varphi_1(\rho) = \varphi_2(\rho)$ which means that
$(\varphi_1(\rho), \rho)$ is a solution to~\eqref{sys:diff}.

The FEM is used to discretize both equations, \eqref{sys:vf1} and
\eqref{sys:vf2}.  To give additional details, let us fix $V_h(\Omega)$
a finite-dimensional subspace of $H^1(\Omega) \cap L^\infty(\Omega)$
and
$V_{h,\gammad}(\Omega) \coloneq \{u_h \in V_h(\Omega) \mid u_h = 0
\text{ on } \gammad\}$, both equipped with the $H^1$-norm.
Let $\rho_h \in V_h(\Omega)$ be such that $\rho_h \ge 0$.  The variational formulation of the
discrete approximation of \eqref{sys:vf1} consists in finding
$\varphi_{1,h} \in V_{h, \gammad}(\Omega)$ such that
\begin{equation}\label{eq : vf_sys1}
  \forall \chi \in V_{h, \gammad}(\Omega), \qquad
  \int_\Omega \nabla \varphi_{1,h} \nabla \chi \intd x
  = \int_\Omega \rho_h \chi \intd x
  + \int_{\gammac} A \chi \intd x .
\end{equation}
We proceed similarly for \eqref{sys:vf2}. Let
$V_{h,\partial\Omega}(\Omega) \coloneq \{u_h \in V_h(\Omega) \mid u_h = 0
\text{ on } \partial \Omega \}$. 
The variational form of
\eqref{sys:vf2} consists in seeking
$\varphi_{2,h} \in V_{h,\gammad}(\Omega)$ such that
$\varphi_{2,h} = 1$ on $\gammac$ and
\begin{equation}\label{eq: sys_vf2}
  \forall \chi \in V_{h,\partial\Omega}(\Omega),\qquad
  \int_\Omega (\rho_h+1)\nabla \varphi_{2,h} \nabla \chi \intd x
  = \int_\Omega \rho_h \chi \intd x.
\end{equation}
The discrete version of functional $J$ defined by~\eqref{def : J} is given by
\begin{equation*}
  J_h: V_h(\Omega) \rightarrow \IR: \rho_h \mapsto
  \tfrac{1}{2} \bigl\|\nabla
  \bigl(\varphi_{1,h}(\rho_h)-\varphi_{2,h}(\rho_h) \bigr) \bigr\|^2_\Omega.
\end{equation*}
In order to minimize $J_h$, we want to calculate its
gradient $\nabla J_h \in V_h(\Omega)$ with respect to the $H^1$-topology, which is characterized as follows. For any
$z \in V_h(\Omega)$,
\begin{align}
  \bigl(\nabla J_h(\rho_h) \bigm| z\bigr)_{H^1}
  &= \frac{\partial J_h}{\partial \rho_h}(\rho_h)[z] \nonumber \\
  & = \bigl(\nabla(\varphi_{1,h}-\varphi_{2,h}) \bigm|
    \nabla\bigl(\partial_{\rho_h}
    (\varphi_{1,h}-\varphi_{2,h})[z]\bigr) \bigr)_\Omega \nonumber\\
  & = \bigl( \nabla(\varphi_{1,h}-\varphi_{2,h}) \bigm|
    \nabla (\varphi'_{1,h,z}-\varphi'_{2,h,z}) \bigr)_\Omega , \label{eq : grad_J}
\end{align}
where $\varphi'_{1,h,z} \coloneq \partial_{\rho_h} \varphi_{1,h}[z]$
and $\varphi'_{2,h,z} \coloneq \partial_{\rho_h} \varphi_{2,h}[z]$.  To
compute $\varphi'_{1,h,z}$ and $\varphi'_{2,h,z}$, we
differentiate Equations~\eqref{eq : vf_sys1} and \eqref{eq:
    sys_vf2} with respect to $\rho_h$ in the direction~$z$.  We
then obtain the following two equations:
\begin{equation}\label{eq : var_phi_prim_1z}
  \forall \chi \in V_{h,\gammad}(\Omega), \qquad
  \int_\Omega \nabla \varphi'_{1,h,z} \nabla \chi \intd x
  = \int_\Omega z \chi \intd x,
\end{equation}
where $\varphi'_{1,h,z} \in V_{h,\gammad}(\Omega)$ and
\begin{equation} \label{eq : var_phi_prim_2z}
  \forall \chi \in V_{h,\partial\Omega}(\Omega), \qquad
  \int_\Omega  (\rho_h +1) \nabla \varphi'_{2,h,z} \nabla \chi \intd x
  = \int_\Omega z \chi - z \nabla \varphi_{2,h} \nabla \chi \intd x,
\end{equation}
where $\varphi'_{2,h,z} \in V_{h,\partial\Omega}(\Omega)$.

We now want to rewrite the right hand side of
Equality~\eqref{eq : grad_J} in order to
not to have to compute it for a basis of $z$.  For the part involving
$\varphi'_{1,h,z}$, we have
\begin{align*}
  \bigl( \nabla(\varphi_{1,h}-\varphi_{2,h}) \bigm|
  \nabla \varphi'_{1,h,z} \bigr)_\Omega
  & = \int_\Omega \nabla(\varphi_{1,h}-\varphi_{2,h})
    \nabla \varphi'_{1,h,z} \intd x \\
  & = \int_\Omega (\varphi_{1,h}-\varphi_{2,h}) z \intd x
    = \bigl(\varphi_{1,h}-\varphi_{2,h} \bigm| z\bigr)_\Omega
    \quad \text{by \eqref{eq : var_phi_prim_1z}} .
\end{align*}
For the second part, we introduce the intermediate function
$\Psi_h \in V_{h,\partial\Omega}(\Omega)$ which is the solution to the
following variational problem:
\begin{equation} \label{eq : psi}
  \forall \chi \in V_{h,\partial\Omega}(\Omega), \qquad
  \int_\Omega (\rho_h +1) \nabla \Psi_h \nabla \chi \intd x
  = \int_\Omega \nabla (\varphi_{1,h}-\varphi_{2,h}) \nabla \chi \intd x .
\end{equation}
Hence, we have:
\begin{align*}
  \bigl( \nabla(\varphi_{1,h}-\varphi_{2,h}) \bigm|
  \nabla \varphi'_{2,h,z} \bigr)_\Omega
  & = \int_\Omega \nabla(\varphi_{1,h}-\varphi_{2,h})
    \nabla \varphi'_{2,h,z} \intd x \\
  & = \int_\Omega (\rho_h +1) \nabla \Psi_h  \nabla \varphi'_{2,h,z}\intd x
  &&\text{by \eqref{eq : psi}}\\
  & = \int_\Omega z \Psi_h - z \nabla \varphi_{2,h} \nabla \Psi_h \intd x
  && \text{ by \eqref{eq : var_phi_prim_2z}} \\
  & = \bigl(\Psi_h
    - \nabla \varphi_{2,h} \nabla \Psi_h \bigm| z \bigr)_\Omega.
\end{align*}
The last identities allow to simplify the equation \eqref{eq : grad_J}
for $\nabla J_h(\rho_h) \in V_h(\Omega)$ into
\begin{equation}\label{eq : grad_J_fin}
  \forall z \in V_h(\Omega),\quad
  \bigl(\nabla J_h(\rho_h) \bigm| z\bigr)_{H^1}
  = \bigl(\varphi_{1,h}-\varphi_{2,h}
  - \Psi_h
  + \nabla \varphi_{2,h} \nabla \Psi_h \bigm| z\bigr)_\Omega.
\end{equation}

\begin{algorithm}
  For the algorithm, we have created four subroutines as follows:
  given $\rho \in V_h(\Omega)$, we can
  \begin{enumerate}
  \item Compute the solution $\varphi_{1,h}(\rho_h)$ to \eqref{eq :
      vf_sys1} using the FEM.
  \item Compute the solution $\varphi_{2,h}(\rho_h)$ to \eqref{eq:
      sys_vf2} with the FEM.
  \item Compute the functional $J_h(\rho_h)$ with
    $\varphi_{1,h}(\rho_h)$ and $\varphi_{2,h}(\rho_h)$.
  \item Compute $\nabla J_h(\rho_h)$ by solving
    Equation~\eqref{eq : grad_J_fin} with the FEM, the quantities
    $\varphi_{1,h}(\rho_h)$ and $\varphi_{2,h}(\rho_h)$ being computed
    using the above subroutines and $\Psi_h$ being the solution
    to~\eqref{eq : psi}, computed again with the FEM.
  \end{enumerate}
  Given a tolerance \texttt{tol} (by default $10^{-12}$),
  a maximum number of iterations
  \texttt{N} (by default $500$)
  and an initial guess $\rho_0 \in V_h(\Omega)$,
  we use a minimization algorithm that stops when $J_h(\rho_h) \le
  \texttt{tol}$ or the maximum number of iterations exceeds \texttt{N}.
\end{algorithm}

\subsection{Some additional numerical details}
The Finite Element Method, used to compute the functions $\varphi_{1,h}$ and $\varphi_{2,h}$, is implemented via the Python library \textit{Netgen} \cite{Netgen}. Here we choose $V_h(\Omega)$ to be the space of P1 elements.  For the minimization step, we use the \textit{Scipy} library, and more specifically its function \textit{scipy.optimize.minimize}, which provides access to several optimization algorithms. Among these, we selected only methods that require the gradient of the objective function but not the explicit computation of the Hessian matrix.

Of those routines, the L-BFGS-B one is significantly the faster but exhibits convergence problems depending on the version of Scipy used. In the latest version, the method fails due to a known unresolved bug \cite{github_issue}. In earlier version, we observed some abnormal terminations when applying L-BFGS-B to the Rosenbrock function. The SLSQP method is the slowest of all those tested (up to 50 times slower).

The remaining two methods, namely Trust Constraint (trust-constr) and Conjugate Gradient (CG), are the fastest and converge to the default tolerance of the functional $J_h$ across a variety of test cases. They are essentially on par in terms of speed and precision. The numerical tests below were performed using trust-constr.

\subsection{Numerical results}

\subsubsection{Radial case}

Since analytical solutions are available for the radial case, we first
test our algorithm in this setting to check whether it
converges to the exact solution.  We take the radius of the
interior boundary, $r_0$, as $0.25$ (see Figure~\ref{fig : radial})
and the Neumann condition is given
by formula~\eqref{lemma : A_radial} with $\lambda=-50$.
Starting with the initial density $\rho_0(x, y) \coloneq x^2+y^2+1$,
we let the algorithm converges for various mesh sizes.
Figure~\ref{fig : cvg_radial} shows, in a double logarithmic scale, the relative error between
$\varphi_{\text{exact}}$, the exact solution $\varphi$ given
by \eqref{eq : sol_radial_phi}, and $\varphi_{\text{num}}$,
the approximate solution $\varphi$ computed numerically, in blue.
The same figure shows,
in red, the relative error between $\rho_{\text{exact}}$,
the exact solution $\rho$ given by \eqref{eq :
  sol_radial_rho}, and $\rho_{\text{num}}$, the approximate
solution $\rho$ computed numerically. Both errors decrease to zero as the
mesh size $h \to 0$.  This provides evidence that,
in that specific case, the numerical solution approaches the exact
one with an error of order of about $h^{1.66}$ for $\varphi$ and $h^{2.33}$ for $\rho$.

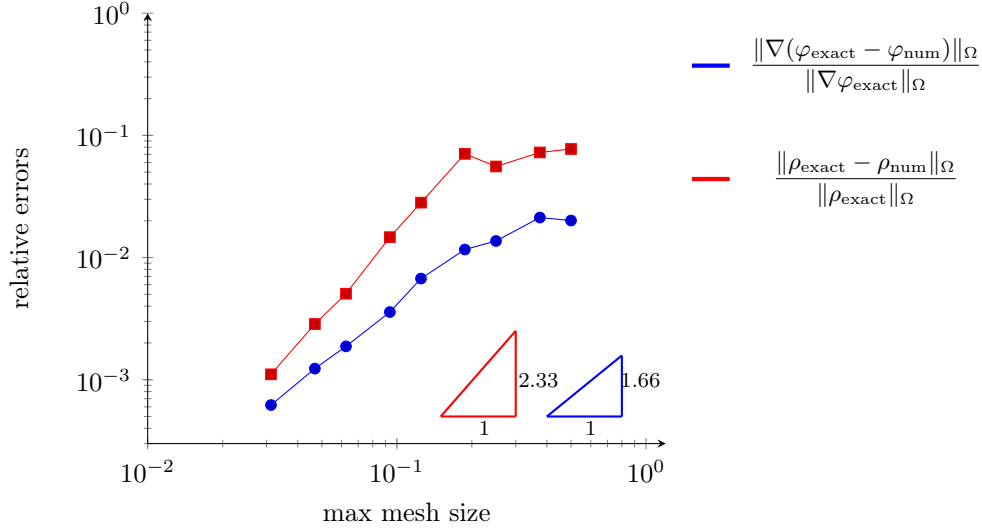
\begin{figure}[h]
  \centering
  \begin{tikzpicture}
    \begin{axis}[
                axis x line = bottom,
        axis y line = left,
        tick label style={font=\footnotesize},
        ylabel style={yshift=0.35cm},
        xlabel={\footnotesize max mesh size},
        ylabel={\footnotesize relative errors},
        xmin = 0.01,xmax=1.2,
        ymin = 0.0003,ymax = 1,
        xmode=log,         ymode=log,         log basis x=10,         log basis y=10       ]
            \addplot table {\data{cvg_phi.dat}};
      \addplot table {\data{cvg_rho.dat}};
      
            \draw[red, thick] (axis cs:0.15,0.0005) -- (axis cs:0.3,0.0005) node[black] at (axis cs: 0.225,0.0004) {\tiny $1$};
      \draw[red, thick] (axis cs:0.3,0.0005) -- (axis cs:0.3,0.002514) node[black] at (axis cs:0.37,0.001) {\tiny $2.33$};
      \draw[red, thick] (axis cs:0.15,0.0005) -- (axis cs:0.3,0.002514);

            \draw[blue, thick] (axis cs:0.4,0.0005) -- (axis cs:0.8,0.0005)
      node[black] at (0.6,0.0004) {\tiny $1$};
      \draw[blue, thick] (axis cs:0.8,0.0005) -- (axis cs:0.8,0.001580)
      node[black] at (0.95, 0.001) {\tiny $1.66$};
      \draw[blue, thick] (axis cs:0.4,0.0005) -- (axis cs:0.8,0.001580);
    \end{axis}
        \draw[blue, ultra thick] (7.2,5) -- (7.7,5) node[black] at (9.5,5) {      \scriptsize $\displaystyle
      \frac{\|\nabla (\varphi_{\text{exact}}-\varphi_{\text{num}})\|_\Omega}{
        \|\nabla \varphi_{\text{exact}}\|_\Omega}$};
    \draw[red, ultra thick] (7.2,3.5) -- (7.7,3.5) node[black] at (9.5,3.5) {      \scriptsize $\displaystyle
      \frac{\|\rho_{\text{exact}}-\rho_{\text{num}}\|_\Omega}{\| \rho_{\text{exact}}\|_\Omega}$};

              \end{tikzpicture}
  \caption{Convergence in the radial case to the exact solutions.}
  \label{fig : cvg_radial}
\end{figure}

To shed some light on the size of the basin of convergence, we ran the
algorithm with different initial densities $\rho_0$ on a mesh with a
size $h \approx 3\cdot 10^{-2}$ ($2690$ degrees of freedom).  For the
above choice of $\rho_0$, the initial value of $J_h$ is
$3.6 \cdot 10^{-6}$ and the default tolerance of $10^{-12}$ is reached
after $152$ iterations.  For the nonradial $\rho_0(x,y) = x^2+100$,
the initial value of $J_h$ is approximately $430$ and the algorithm
reaches the default tolerance after $157$ iterations.  We have also
tried several additional initial $\rho_0$, for example one that oscillates
in the polar angle, $\rho_0(x,y) = 1 + y/\sqrt{x^2 + y^2}$, and the
results are sensibly the same in all cases.  This highlights the
robustness of the algorithm with respect to the initial guess.

In section~\ref{s:radial}, all radial solutions were determined.  A
natural question is whether nonradial solutions exist for some
boundary data $A: \gammac \to \IR$.  Unfortunately, the results of
section~\ref{s:unipolar} do not provide an answer to that question.
As an element of evidence towards a positive answer, we chose the
nonradial $A(x,y) = \frac{1}{2} + \frac{1}{4}x/\sqrt{x^2 + y^2}$ (the
constraint \eqref{eq : ineg_deriv_normal}, with
$\partial\phie/\partial\normal$ readily computed from~\eqref{eq :
  phie_radiale}, is satisfied),
the initial guess $\rho_0(x,y) = x^2+y^2+1$ and ran our algorithm.  It
converged with a final value for $J_h$ of approximately
$8.5 \cdot 10^{-13}$, suggesting that a nonradial solution indeed
exists.  The level curves of the final $\varphi$ and $\rho$ are
depicted in Figure~\ref{fig : phi_rho_Acos}. 

Finally, as we can see in the previous section \ref{s:radial}, there does not exist a solution for $A$ greater than $A_{-1}$. So, in that case, the algorithm should not converge. To exemplify this, let us take $A = 5$ and $\rho_0 = x^2+y^2+1$. The algorithm stops because the default number of iterations is exceeded and the final value of $J_h$ is approximately $1.9$, which is not small enough to consider the returned $(\varphi, \rho)$ to be a solution.

\begin{figure}[h!t]
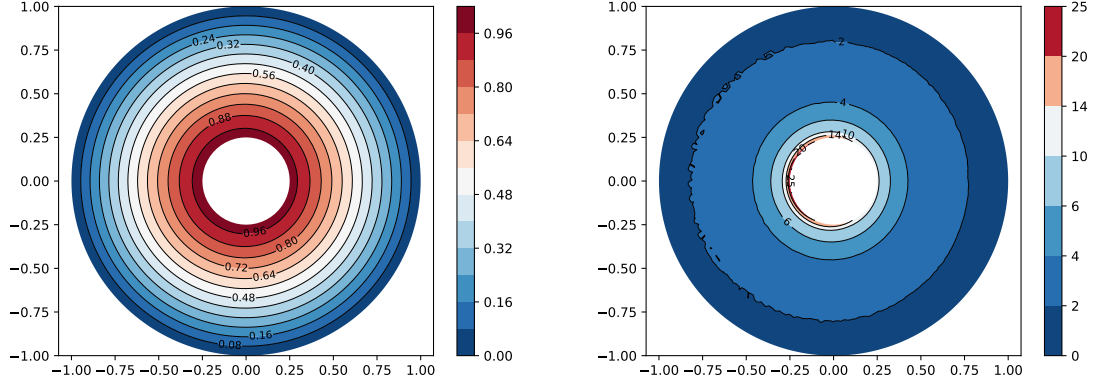

  \begin{minipage}[c]{.49\linewidth}
    \centering
    \includegraphics[scale=1,width=8cm]{\data{A_sin_radial_phi.pdf}}
  \end{minipage}
  \hfill  \begin{minipage}[c]{.49\linewidth}
    \centering
    \includegraphics[scale=1,width=8cm]{\data{A_sin_radial_rho.pdf}}
  \end{minipage}

  \caption{Levels curves for the numerical solutions $\varphi$ (left) and
    $\rho$ (right) with $A(x,y) = \frac{1}{2} + \frac{1}{4}x/\sqrt{y^2+x^2}$.
    \label{fig : phi_rho_Acos}}
\end{figure}

\subsubsection{The unipolar half-disk}
\label{sec:half-disk}

Now, let us consider a more general case for which exact solutions
cannot be computed analytically.  More precisely, we consider the case
of a half-disk containing a circular conductor (see Figure \ref{fig :
  half-disk}).  The radius of the interior circle centered at $(0,0)$
(i.e.\ the interior boundary) has been chosen to be $0.25$.  The
center of the exterior circle is at the point $(0,-1)$ and its radius
is $2$.  Since we only consider the half-disk, the coordinates of its
corners are $(-2,-1)$ and $(2,-1)$.

As we did for the radial case, we have tested different initial
guesses $\rho_0$ and different values for the Neumann condition $A$.
We performed our tests with a mesh of size $h = 5\cdot 10^{-2}$
($2784$ degrees of freedom).  First, we considered the initial guess
$\rho_0(x,y) = \sqrt{x^2+y^2}+1$ and the Neumann condition $A=1$ and
we numerically verified that the constraint \eqref{eq :
  ineg_deriv_normal} is satisfied.
In Figure \ref{fig : courbe_phi_rho}, we can see the level curves of
the solution $(\varphi,\rho)$.  We can note that the value 
for $\rho$ is higher around the interior boundary. This makes sense
physically because $\rho$ represents the space charges in the air that
in practice are higher near the conductor. We also observe that $\rho$
goes to zero near the~corners.
As in the previous section, the algorithm keep converging (to the same
solution) when starting with different initial guesses more distant
from the solution (i.e.\ with higher values of $J_h$).

\begin{figure}[ht]
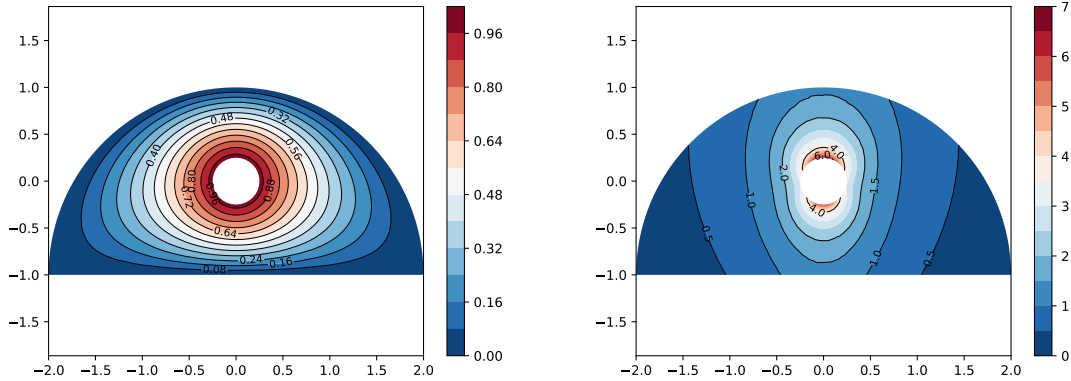

  \begin{minipage}[c]{.49\linewidth}
    \centering
    \includegraphics[scale=1,width=8cm]{      \data{sqrt_xcarre_ycarre_plus1_phi.pdf}}
  \end{minipage}
  \hfill  \begin{minipage}[c]{.49\linewidth}
    \centering
    \includegraphics[scale=1,width=8cm]{      \data{sqrt_xcarre_ycarre_plus1_rho.pdf}}
  \end{minipage}
  
  \caption{Levels curves for $\varphi$ (left) and $\rho$ (right)
    with $A = 1$.
    \label{fig : courbe_phi_rho}}
\end{figure}

We also again tested the convergence of the algorithm for non-constant
$A$'s.  As an example, for $A(x,y) = x/2 + 0.5$ and
$\rho_0(x,y) = \sqrt{x^2+y^2}+1$, the algorithm reaches the default
tolerance after $56$ iterations.  The final value of $J_h$ is
approximately $9.16 \cdot 10^{-13}$, suggesting that such a
solution exists.  We can see on Figure~\ref{fig : A non constant} the
level curves of the solution $(\varphi, \rho)$.

\begin{figure}[h!]
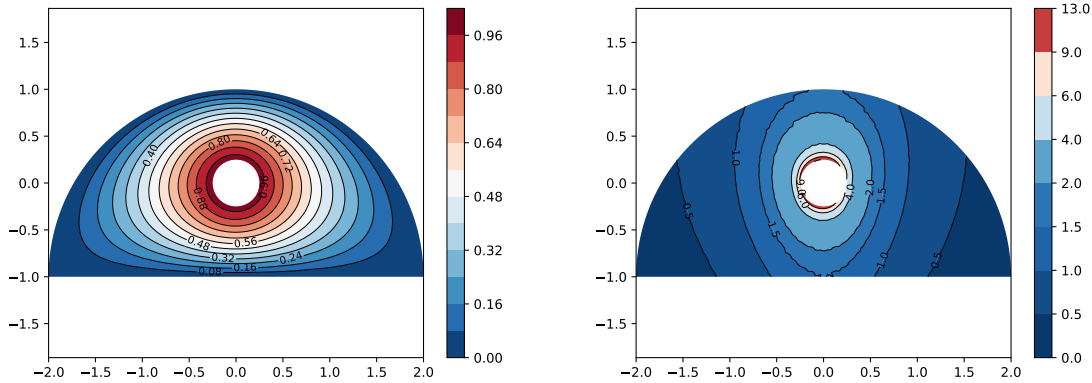

  \begin{minipage}[c]{.49\linewidth}
    \centering
    \includegraphics[scale=1,width=8cm]{\dataSub{A_xdemi}{phi_niveau.pdf}}
  \end{minipage}
  \hfill  \begin{minipage}[c]{.49\linewidth}
    \centering
    \includegraphics[scale=1,width=8cm]{\dataSub{A_xdemi}{rho_niveau.pdf}}
  \end{minipage}

  \caption{Levels curves for $\varphi$ (left) and $\rho$ (right) with
    $A(x,y) = x/2 +0.5$.
    \label{fig : A non constant}}
\end{figure}

Finally, for a value of $A$ greater than
$\partial\phie/\partial\normal$ (see Theorem~\ref{thm :
  deriv_normal}), the algorithm should not converge.  To exemplify
this, let us take $A = 3$ and $\rho_0 = \sqrt{x^2+y^2}+1$.  The
algorithm stops because the default maximum number of iterations is
exceeded and the final value of $J_h$ is approximately
$4.32\cdot 10^{-3}$, which is not small enough to consider the
returned $(\varphi, \rho)$ to be a solution.

\begin{remark}
  In these experiments, we consider simplified test cases which allow
  for an initial evaluation of the algorithm’s performance. Although
  these cases are not representative of realistic physical scenarios,
  extending the approach to more representative configurations is part
  of our future work.
\end{remark}

\section*{Acknowledgements}

We are grateful to Prof.~Stéphane Clénet (ENSAM Lille, France) for introducing us to the topic of HVDC transport, for his invaluable help on the physical aspects of the problem, as well as for  his enthusiastic support of our findings.
We also would like to thank Zuqi Tang (L2EP Lille, France) for interesting discussions and for his involvement in the supervision of Madeline Chauvier's PhD thesis.
\printbibliography

\end{document}